\documentclass[12pt,twoside,reqno]{amsart}
\usepackage[colorlinks=true,citecolor=blue]{hyperref}
\usepackage{mathptmx, amsmath, amssymb, amsfonts, amsthm, mathptmx, enumerate, color}
\usepackage{graphicx}
\usepackage{epstopdf}

\newtheorem{theorem}{Theorem}[section]

\newtheorem{lemma}{Lemma}[section]
\newtheorem{proposition}{Proposition}[section]
\theoremstyle{definition}
\newtheorem{definition}{Definition}[section]

\newtheorem{remark}{Remark}[section]

\numberwithin{equation}{section}

\usepackage{cite}
\newcommand{\Real}{\mathbb R}
\newcommand{\real}[1]{{\mathbb R}^{#1}}
\newcommand{\set}[1]{\left\{#1\right\}}

\newcommand{\U}{\mathbb{U}}

\begin{document}
\setcounter{page}{1}

\vspace*{1.0cm}
\title[The Natural Physics of Optimization]
{On The Mathematics of the Natural Physics of Optimization$^{*}$}
\author[I. M. Ross]{ I. M. Ross$^{\dag}$}
\maketitle
\vspace*{-0.6cm}

\begin{center}
{\footnotesize {\it

Department of Mechanical and Aerospace Engineering, Naval Postgraduate School, Monterey, CA\\

}}\end{center}

\vskip 4mm {\small\noindent {\bf Abstract.}
A number of optimization algorithms have been inspired by the physics of Newtonian motion. Here, we ask the question: do algorithms themselves obey some ``natural laws of motion,'' and can they be derived by an application of these laws?  We explore this question by positing the theory that optimization algorithms may be considered as some manifestation of hidden algorithm primitives that obey certain universal non-Newtonian dynamics.  This natural physics of optimization is developed by equating the terminal transversality conditions of an optimal control problem to the generalized Karush/John-Kuhn-Tucker conditions of an optimization problem. Through this equivalence formulation, the data functions of a given constrained optimization problem generate a natural vector field that permeates an entire hidden space with information on the optimality conditions. An ``action-at-a-distance'' operation via a Pontryagin-type minimum principle produces a local action to deliver a globalized result by way of a Hamilton-Jacobi inequality. An inverse-optimal algorithm is generated by performing control jumps that dissipate quantized ``energy'' defined by a search Lyapunov function. Illustrative applications of the proposed theory show that a large number of algorithms can be generated and explained in terms of the new mathematical physics of optimization.

 }

\renewcommand{\thefootnote}{}
\footnotetext{\textit{$^*$J. Nonlinear Var. Anal. 10 (2026), 661-686. https://doi.org/10.23952/jnva.10.2026.3; special issue dedicated to Yurii Nesterov on the occasion of his 70th birthday.}
\par
$^\dag$Distinguished Professor
}

\section{Introduction}\label{sec:intro}

For the limited purposes of introducing a key idea, we use the notation $x_f \in \real{n}, \ n \in \mathbb{N}_+$ to be a continuous variable in the constrained optimization problem given abstractly by,
\begin{eqnarray}
&(O_A) \left\{
\begin{array}{lrl}
\text{minimize } && g_0(x_f)  \\
\text{subject to}
&& x_f \in  C
\end{array} \right.& \label{eq:OP-A}
\end{eqnarray}
where $C \subseteq \real{n}$ is a continuous set and $g_0: x_f \mapsto \Real$ is a Lipschitz-continuous function. The first-order necessary condition for Problem~$(O_A)$ is given by the (generalized) Karush/John-Kuhn-Tucker condition\cite{clarke-green-book,clarke-2013book,vinter,boris-book-2006}:
\begin{equation}\label{eq:JKKT-abstract}
0 \in \nu_0\,\partial g_0(x_f) + N_C(x_f)
\end{equation}
where, $\nu_0 \in \Real_+$ is a Fritz John cost multiplier\cite{bazaraa-2006,luenberger-2008,NW:NumOptBook}, $\partial g_0(x_f)$ is the (limiting/Mordukhovich) subdifferential\cite{vinter,boris-book-2006, clarke-2013book} of $g_0$ at $x_f$ and $N_C(x_f)$ is the (limiting/Mordukhovich) normal cone\cite{vinter,boris-book-2006,clarke-2013book} to $C$ at $x_f$. Next, consider an optimal control problem given by,
\begin{eqnarray}
&(\widetilde{O_B}) \left\{
\begin{array}{lrl}
\text{minimize } && g_0(x(t_f))  \\
\text{subject to} && \dot x(t)= \widetilde{f}(x(t), \widetilde{u}(t))\\
&& x(t_f) \in  C \\
&&x(t_0) = x^0 \\
\end{array} \right.& \label{eq:prob-Btilde}
\end{eqnarray}
where, $t \in \Real$ is an independent ``time'' variable, $x(t) \in \real{n}$ is a state variable,  $\widetilde{u}(t) \in \real{l}$ is a control variable, $\widetilde{f}: \real{n} \times \real{l} \to \real{n}$ is a controllable vector field that satisfies the standard hypothesis\cite{clarke-2013book}, and $x^0 \in \real{n}$ is a given initial point at time $ t = t_0 \le t_f$.  The unknowns in Problem~$(\widetilde{O_B})$ are the state-control function pair, $ t \mapsto (x(t), \widetilde{u}(t)) $ and $t_f < \infty$. The terminal transversality condition for Problem~$(\widetilde{O_B})$ is given by\cite{vinter,clarke-2013book},
\begin{equation}\label{eq:tvc-cone}
\lambda(t_f) \in \widetilde{\nu}_0\,\partial g_0(x(t_f)) + N_C(x(t_f))
\end{equation}
where, $\lambda(t_f) \in \real{n}$ is the final-time value of the adjoint covector and $\widetilde{\nu}_0 \in \Real_+$ is the cost multiplier associated with \eqref{eq:prob-Btilde}. Now suppose there exists a Problem~$(O_B)$ whose extremal solution is such that
\begin{equation}\label{eq:keyDef4ProbB}
\lambda(t_f) = 0
\end{equation}
Then it is obvious that the following statements are true:
\begin{enumerate}
\item The final-time value $x(t_f)$ of an extremal state trajectory to Problem~$(O_B)$ is a candidate solution to Problem~$(O_A)$;
\item An extremal state trajectory $t \mapsto x(t)$ to Problem~$(O_B)$ is a continuous-time convergent ``algorithm'' for Problem~$(O_A)$ with $x^0$ as a starting point (i.e., $x^0$ is a guess of a solution to Problem~$(O_A)$); and,
\item An extremal control trajectory  $[t_0, t_f] \ni t \mapsto \widetilde{u}(t)$ to Problem~$(O_B)$ generates a continuous-time convergent algorithm for Problem~$(O_A)$ via a solution to the initial value problem (IVP) given by,
    \begin{equation}\label{eq:IVPalg4probA}
    \dot x= \widetilde{f}(x, \widetilde{u}(t)), \quad x(t_0) = x^0
    \end{equation}
\end{enumerate}
%
%--------------------
\begin{definition}[\emph{Hidden Algorithm Primitive for Problem~$(O_A)$}]\label{def:algor-primitive}
Given a point $x^0 \in \real{n}$, let $[t_0, t_f] \ni t \mapsto x(t)$ be a solution to Problem~$(O_B)$.  Then the point-to-set map,
\begin{equation}\label{eq:algorithm-primitive-def}
x^0 \mapsto \set{x^0 = x(t_0), [t_0, t_f] \ni t \mapsto x(t)}
\end{equation}
is called a hidden algorithm primitive for Problem~$(O_A)$.
\end{definition}
%---------------------
Definition~\ref{def:algor-primitive} assumes the existence of Problem~$(O_B)$. The existence of this problem is given by the Transversality Mapping Principle\cite{ross:CD,rossJCAM-1,rossJCAM-2}.
%
%==============
\begin{theorem}[Transversality Mapping Principle\cite{ross:CD,rossJCAM-1,rossJCAM-2}]
There exists a Problem~$(O_B)$ whose terminal transversality condition is given by \eqref{eq:keyDef4ProbB}.
\end{theorem}
%==================
\begin{remark}
By definition, a hidden algorithm primitive is convergent to a candidate optimal solution to Problem~$(O_A)$.
\end{remark}
%=================
%
It is apparent that a discretization of a hidden algorithm primitive produces an algorithm.  However, \emph{it is extremely important to note at the very outset that we do not propose to produce algorithms by generating differential equations and discretizing them afterwards}; rather, it will be apparent as we develop the ideas that the eventual formulas for generating algorithms do not directly depend upon producing hidden algorithm primitives.  Yet, the concept of a hidden algorithm primitive will be pervasive to the logical process of understanding the natural mathematical physics of optimization articulated in terms of a Hamilton-Jacobi theory.  Once these equations are developed, we will combine the ideas of inverse optimality\cite{glad87,freeman} and proximal aiming\cite{prox-aiming-1994,clarkeLyap} to produce algorithms without ever generating hidden algorithm primitives.

\section{A Preliminary Discussion of the Proposed Ideas}\label{sec:prelim-discuss}
As a stand-alone concept, the transversality mapping principle inverts conventional wisdom.  That is, it seems to suggest that optimal control theory can be used to solve an optimization problem instead of the other way around as is conventionally presumed\cite{bazaraa-2006,NW:NumOptBook,luenberger-2008}. Indeed, in the absence of additional information, this inversion seems ill-advised.  All of these undesirable motivational features are acknowledged in \cite{rossJCAM-1} wherein the initial ideas were first formulated. However, as noted at the end of Section~\ref{sec:intro}, we will never generate an extremal state trajectory $t \mapsto x(t)$ for Problem~$(O_B)$; yet, its idea will be pervasive and critical to the production of a practical algorithm for solving Problem~$(O_A)$. This is why a hidden algorithm primitive is indeed hidden.

It is critically important to note at this juncture that the central feature of the ideas introduced in Section~\ref{sec:intro} is \emph{not} to take the limits of existing algorithms to generate ordinary differential equations (ODEs) and/or to modify the resulting ODEs to improve algorithmic efficiency. This is a subject of a large body of work\cite{boggs71, smale76, brown+biggs, gavurin, polyak64, NLP2ODE-1980,NLP2ODE-1981,NLP2ODE-1989,NLP2ODE-1994,NLP2ODE-2006,NLP2ODE-2007,NLP2ODE-2017, ODEinML-2016-1,ODEinML-2016-2,ODEinML-2016-3, ODEinML-2019,ODEinML-2021-1,ODEinML-2021-2,ODEinML-2021-3, ODEinML-2021-4,ODEinML-2021-5,lemarechal-cauchy-grad-2012,curry-cauchy-grad-1944} with many interesting side stories. For instance, in 1958, Gavurin\cite{gavurin} produced the continuous Newton-flow equation by taking the limit of Newton's method and analyzing the resulting ODE; see also Smale\cite{smale76}.  Over a hundred years earlier, Cauchy (according to Lemar\'{e}chal\cite{lemarechal-cauchy-grad-2012}) performed the opposite: he proposed the gradient method based on the continuous gradient-flow equation\cite{curry-cauchy-grad-1944}. 
Recent renewed interest in using ODEs as models for optimization algorithms can be broken down into two categories: one that is focused on solving nonlinear programming problems\cite{NLP2ODE-1980,NLP2ODE-1981,NLP2ODE-1989,NLP2ODE-1994,NLP2ODE-2006,NLP2ODE-2007,NLP2ODE-2017}  and the other that is largely concentrated on ``accelerated'' optimization algorithms for unconstrained or convex optimization problems\cite{ODEinML-2016-1,ODEinML-2016-2,ODEinML-2019,ODEinML-2021-1,ODEinML-2021-2,ODEinML-2021-3, ODEinML-2021-4,ODEinML-2021-5}.  In many ways, the ongoing spike in interest in using ODEs to explain or improve accelerated optimization algorithms can be considered as an evolution of the early ideas\cite{gavurin, polyak64, boggs71, smale76,curry-cauchy-grad-1944,lemarechal-cauchy-grad-2012,brown+biggs}  that lay somewhat dormant prior to the machine learning revolution\cite{bengio:deepLearning,MLopt:SIREV}.  It is apparent from this brief review of the rich literature that the concepts introduced in Section~\ref{sec:intro} are sharply different from these prior works in almost all aspects.  To drive home this point, we explicitly identify some of these fundamental differences as follows:
\begin{enumerate}
\item The key idea of equating the first-order necessary conditions of an optimization problem to the transversality condition of an optimal control problem was first initiated in \cite{rossJCAM-1} and represents a marked departure from all prior works.  This paper is an advancement of the theory introduced in \cite{rossJCAM-1}.

\item The proposed theory \emph{never} employs taking limits of existing algorithm, as, for example, in \cite{smale76,ODEinML-2016-1,ODEinML-2021-5}; rather, it is based on an independent production of algorithms without using any a priori knowledge of existing ones.  This is part of the reason behind Definition~\ref{def:algor-primitive}.
\item As presented in Section~\ref{sec:intro}, sheer intellectual curiosity pertaining to \eqref{eq:keyDef4ProbB} is the primary motivation rather than a need to produce a new theory for optimization.  The fact that this intellectual curiosity leads to a new and natural ``physics'' of optimization (as developed in the remainder of this paper) is an a posteriori result.
\item It is evident that the ideas presented in Section~\ref{sec:intro} are not limited to unconstrained optimization or convex optimization or first- or second-order methods\cite{ODEinML-2016-1,ODEinML-2016-2,ODEinML-2019,ODEinML-2021-1,ODEinML-2021-2,ODEinML-2021-3, ODEinML-2021-4,ODEinML-2021-5}.  It will be apparent later that \eqref{eq:IVPalg4probA} in the ensuing theory is merely a projection of a larger set of controllable differential equations. Hence, there is also no presumption of an order in terms of the order of an ODE.
\item As shown in \cite{rossJCAM-2}, the proposed theory explains the ``mystery'' surrounding Nesterov's accelerated gradient algorithm\cite{nesterov83}.  This explanation is not achieved by taking the limits of Nesterov's original algorithm\cite{nesterov83} to generate ODEs\cite{ODEinML-2016-1,ODEinML-2021-5}; rather, Nesterov's algorithm is shown to be a particular consequence of \eqref{eq:keyDef4ProbB} and an unconstrained version of Problem~$(O_B)$.  In other words, Nesterov's algorithm is \emph{derived} in \cite{rossJCAM-2} and not assumed.  This derivation is based on seeking $x(\cdot) \in W^{2,\infty}([t_0, t_f], \real{n})$ as an engineering technique\cite{ross-book} to reduce the total variation of a hidden algorithm primitive.  See also \cite{ross:accelerated-arxiv} and Section~\ref{sec:acceleration} of this paper.

\item It will be apparent later that the control vector $\widetilde{u}$ of Problem~$(O_B)$ is a proxy for the search vector of an algorithm to solve Problem~$(O_A)$. Suppose this control vector is generated by a feedback law $\widetilde{u} = \widetilde{K}(x, t)$.  Any given feedback control changes the ODE in \eqref{eq:IVPalg4probA} to a different one given by $\dot x = \widetilde{f}(x, \widetilde{K}(x, t))$. This idea of systematically generating ODEs is well-known in control theory\cite{clarke-2013book,vinter,sontag-book} and forms the basis for algorithm generation presented in \cite{NLP2ODE-2006,ODEinML-2016-3}.  These prior ideas of using control theory for optimization do not incorporate \eqref{eq:keyDef4ProbB}; rather, they are based on drawing analogies between control theory and optimization.  Furthermore, as implied earlier, the production of algorithms as proposed in this paper will never require the simulation of $\dot x = \widetilde{f}(x, \widetilde{K}(x, t))$; hence, there is no discretization and propagation of an ODE.
\item It is useful to note at this juncture that any ODE given by $\dot x = f(x,t)$ may be viewed as a control ODE $\dot x = u$ with feedback law $u = f(x,t)$.
    \item It will be apparent later that a search direction for a candidate inverse optimal algorithm is completely determined by two factors:
        \begin{enumerate}
        \item a selection of a state-dependent constraint set $\U$ for the control vector, and
        \item a choice of a search Lyapunov function (SLF) that satisfies a Hamilton-Jacobi inequality.
        \end{enumerate}
        Although Lyapunov functions are widely used in optimization, the notion of an SLF is unique to this paper, albeit a generalization of the concept of a control Lyapunov function used in \cite{rossJCAM-1,ross:CD,rossJCAM-2}. Furthermore, as shown in \cite{rossJCAM-1}, a selection of $\U$ has the effect of metricizing the search space.  See also Section~\ref{sec:examples} of this paper.
    \item The same ideas (outlined in Section~\ref{sec:intro}) applies to both constrained and unconstrained optimization problems. In fact, one of the key strengths of the proposed theory is its universal management of constraints.
\end{enumerate}

A theory that explains the collection of the preceding ideas and more is the subject of this paper.

\section{A Word About Notation}\label{sec:notation}
There is little doubt almost all readers will be dissatisfied with the notation used in this paper. Starting with the very first equation, it is obvious that we did not use the standard notation $x$ for the optimization variable in \eqref{eq:OP-A}.  Doing so would require us to use a nonstandard notation for the state variable $x$ in \eqref{eq:prob-Btilde}. Needless to say, we have chosen in Section~\ref{sec:intro} to use the conventional optimal-control symbol $x$ for the state variable at the price of abandoning convention in optimization. Because the same symbol across diverse disciplines have different conventional interpretations, we note upfront that the notation used in this paper, while customary in some fields, may be atypical in another.  For example, because $x$ and $f(x)$ are one of the most overused symbols, we have chosen to abandon them in favor of $q$ for the state variable, a customary notation for generalized coordinates in physics. Similarly, we use the symbol $L$ for the generalized Lagrangian in optimization and not in the sense of a Lagrange cost in optimal control. In the same spirit, we do not use $\nabla_q$ for the gradient operator as is customary in optimization.  Instead, we use $\partial_q$ for the gradient and $\partial^2_q$ for the Hessian while retaining $\partial$ (without a subscript) to indicate the limiting/Mordukhovich subdifferential (which we use sparingly).  Other aspects of our notation are explained during their first occurrence.

\section{A Hamilton-Jacobi Theory for Optimal Hidden Algorithm Primitives}\label{sec:HJ-theory}
As remarked in Section~\ref{sec:prelim-discuss}, Cauchy's inspiration for creating the gradient algorithm was the physics of a flow (motion). We briefly note that even though a gradient algorithm may be ``obvious'' today, its invention did not come about until more than a hundred years after Newton's method\cite{lemarechal-cauchy-grad-2012,curry-cauchy-grad-1944}.  Yet another concept borrowed from physics is Polyak's momentum method\cite{polyak64} that spawned Nesterov's accelerated gradient algorithm\cite{nesterov83}.  In other words, it is not uncommon to design optimization algorithms using the natural physics of motion. Rather than use Nature's laws of motion, here we aim to seek the more fundamental physics of optimization itself using the ideas introduced in Section~\ref{sec:intro}.  Because basic physics is governed by equations rather than abstract sets, we shall now limit the remainder of this paper to discussions wherein the constraint set $C$ in \eqref{eq:OP-A} is parameterized by continuously differentiable functions.  See Section~\ref{sec:nonsmooth-stuff} for some comments on nonsmooth optimization. In addition, per the remarks of Section~\ref{sec:notation}, it will be apparent shortly that it is notationally convenient to depart from the convention of using $x \in \real{n}$ for the unknown variable.  Hence, for the remainder of this paper, we consider the optimization variable to be part of a larger set of generalized coordinates $q \in \real{N}$, $N > n$, and relabel $x$ to be $q_1 \in \real{n}$.  As a result, \eqref{eq:OP-A} is parameterized and reformulated as,
%
%---------------------------------------------------------
\begin{eqnarray}\label{eq:prob-N}
&(N) \left\{
\begin{array} {lll}
\displaystyle\mathop\text{minimize }_{q_1 \in \real{n}}  & g_0(q_1) \\
\text{subject to} & g^L \le g(q_1) \le g^U
\end{array} \right.&
\end{eqnarray}
%---------------------------------------------------------
%
where $g(q_1):= (g_1(q_1), \ldots, g_{m}(q_1))$ is a constraint function with $m \in \mathbb{N}$ components whose values are restricted to lie between some specified lower and upper bounds given by $g^L \in \real{m}$ and $g^U \in \real{m}$ respectively.  If any component of $g^L$ is equal to $g^U$, then the problem formulation contains an equality constraint.  Consequently, Problem~$(N)$ defines a general problem formulation that permits both equality and inequality constraints.

\subsection{The Natural Dynamics of Hidden Algorithm Primitives}

Using Definition~\ref{def:algor-primitive}, a hidden algorithm primitive for \eqref{eq:prob-N} can be articulated as a continuous-time evolution of $t \mapsto q_1(t)$ from an initial state $q_1(t_0)$ (i.e., a starting point of an algorithm) to a final state $q_1(t_f)$ where the latter is a candidate optimal point (i.e., a point that satisfies the algebraic version of \eqref{eq:JKKT-abstract} for Problem~$(N)$).
%
%=================================
\begin{theorem}[Natural Dynamics of Optimization] \label{thm:dynamics4opt}
Let $q:= (q_0, q_1, q_2, q_3, q_4, q_5) \in \Real \times \real{n} \times \real{m} \times \real{n} \times \real{m} \times \Real$ and $u:= (u_0, u_1, u_2) \in \Real \times \real{n} \times \real{m}$.
Define $L(q_0, q_1, q_2) := q_0\, g_0(q_1) +  \left\langle q_2, g(q_1) \right\rangle$ to be the generalized Lagrangian function\cite{clarke-green-book,bazaraa-2006,luenberger-2008} associated with the data functions of \eqref{eq:prob-N} that are assumed to be twice continuously differentiable.  Then
all hidden algorithm primitives for Problem~$(N)$ are governed by the dynamical equation, $\dot q(t) = F(q(t), u(t))$, where $F(\cdot, \cdot)$ is given explicitly by the right-hand-side of the following differential equation:
%
%-------------------------------------
\begin{equation}\label{eq:q-dynamics}
(D) \left\{
\begin{aligned}
\dot q_0(t) &=u_0(t)\\
\dot q_1(t) & = u_1(t) \\
\dot q_2(t) & = u_2(t)\\
\dot q_3(t)  &= -\big[\partial^2_{q_1}L(q_0(t), q_1(t), q_2(t))\big] \, u_1(t) - \partial_{q_1}L(u_0(t), q_1(t), u_2(t)) \\
\dot q_4(t) &=  \big[\partial_{q_1}g(q_1(t))\big] u_1(t)\\
\dot q_5(t) & = \langle \partial_{q_1} g_0(q_1(t)), u_1(t) \rangle
\end{aligned}
\right.
\end{equation}
%=======================================
%
\end{theorem}
%----------------------------------------------
\begin{proof}
The proof of this theorem follows from \cite{rossJCAM-1} wherein a version of \eqref{eq:q-dynamics} is derived from first principles based on optimal control theory.  Here we provide a different and shorter version of a proof but one that relies on optimization theory.

Differentiating $t \mapsto q_1(t)$ with respect to $t$ and setting its derivative to $u_1(t)$ produces,
\begin{equation}\label{eq:q1dot=u1}
\dot q_1(t) := u_1(t)
\end{equation}
Doing the same for the cost and constraint multipliers (i.e., considering $q_0$ and $q_2$ to be functions of $t$ and differentiating) yields,
\begin{align}
\dot q_0(t) := u_0(t)\\
\dot q_2(t) := u_2(t) \label{eq:q2dot=u2}
\end{align}
Define,
\begin{equation}\label{eq:q3=bydef}
\widetilde{q}_3(t):= -\partial_{q_1}L(q_0(t), q_1(t), q_2(t))
\end{equation}
Differentiating \eqref{eq:q3=bydef} with respect to $t$ and using \eqref{eq:q1dot=u1}--\eqref{eq:q2dot=u2} we get,
\begin{align}
\dot{\widetilde{q}}_3(t) &= -\big[\partial^2_{q_1}L(q_0(t), q_1(t), q_2(t))\big] u_1(t) - \partial_{q_1}g_0(q_1(t)) u_0(t) - \big[\partial_{q_1}g(q_1(t))\big]^T u_2(t)  \nonumber\\
&= -\big[\partial^2_{q_1}L(q_0(t), q_1(t), q_2(t))\big] \, u_1(t) - \partial_{q_1}L(u_0(t), q_1(t), u_2(t))
\end{align}
Hence $\widetilde{q}_3(t)$ and $q_3(t)$ differ by a constant of motion. Similarly, by defining
\begin{equation}\label{eq:q4=bydef}
\widetilde{q}_4(t):= g(q_1(t))
\end{equation}
and differentiating \eqref{eq:q4=bydef} with respect to $t$  we get,
\begin{equation}
\dot{\widetilde{q}}_4(t) =  \big[\partial_{q_1}g(q_1(t))\big] u_1(t)
\end{equation}
As a result, $\widetilde{q}_4(t)$ and $q_4(t)$ differ only by a constant of motion. Finally, we set,
\begin{equation}\label{eq:q5=bydef}
\widetilde{q}_5(t):= g_0(q_1(t))
\end{equation}
Differentiating \eqref{eq:q5=bydef} with respect to $t$ generates,
\begin{equation}
\dot{\widetilde{q}}_5(t)  = \langle \partial_{q_1} g_0(q_1(t)), u_1(t) \rangle
\end{equation}
thereby implying that $\widetilde{q}_5(t)$ and $q_5(t)$ differ by a constant of motion. Thus, the dynamical system given by \eqref{eq:q-dynamics} describes an evolution of $t \mapsto q(t)$ when subject to a control action $t \mapsto u(t)$ where $q(t) = (q_0(t), q_1(t), q_2(t), q_3(t), q_4(t), q_5(t))$ represents, up to a constant of motion, the instantaneous values of the (Fritz John) cost multiplier $q_0$, the optimization variable $q_1$, the constraint (Karush-Kuhn-Tucker) multiplier $q_2$, the (negative of the) gradient of the generalized Lagrangian $q_3$, the constraint function $q_4$, and the cost function $q_5$.  Similarly, the controls $u(t) = (u_0(t), u_1(t), u_2(t)$ represent the instantaneous values of the rates of change in the cost multiplier, the optimization variable and the constraint multiplier respectively. Thus, the dynamical system~$(D)$ represents the evolution of all of the primal and dual variables and functions associated with Problem~$(N)$ and its first-order necessary conditions.

The first-order necessary condition for Problem~$(N)$, stated in terms of the final-time conditions of the dynamical system $(D)$, is given by\cite{luenberger-2008,bazaraa-2006},
\begin{eqnarray}
q(t_f) \in T:= \left\{q:\ q_0 \ge 0,\  (q_0, q_2) \ne (0, 0), \ q_3 = 0, \ \ g^L \le q_4 \le g^U,
 \ \ q_2\, \dag\,q_4 \right\}\label{eq:T-def}
\end{eqnarray}
where, $\dag$ denotes the complementarity conditions defined by\cite{luenberger-2008,bazaraa-2006, ross-book},
\begin{equation}\label{eq:cc}
q_2\, \dag\, q_4 \ \Leftrightarrow \ q_2^i  \left\{
                                                              \begin{array}{llr}
                                                                \le 0 & \hbox{if }  & q_4^i = g^{L}_i\\[0.5em]
                                                                = 0 & \hbox{if }  & g^{L}_i < q_4^i < g^{U}_i \\ [0.5em]
                                                                \ge 0 & \hbox{if } & q_4^i = g^{U}_i \\ [0.5em]
                                                                \textit{unrestricted} & \hbox{if } & g^{L}_i = g^{U}_i
                                                              \end{array}
                                                            \right.
\end{equation}
and the superscript $i=1, \ldots, m$ on $q_2$ and $q_4$ denotes their $i^{th}$ component. These necessary conditions are the Karush/John-Kuhn-Tucker conditions\cite{bazaraa-2006,NW:NumOptBook,luenberger-2008,clarke-green-book}. Because these necessary conditions only involve the variables $q_0,  q_2, q_3$ and $q_4$, it follows that all hidden algorithm primitives satisfy the dynamical equations given by \eqref{eq:q-dynamics}.
\end{proof}
%==========================================
\begin{remark}
The dynamical equations given in \eqref{eq:q-dynamics} are contained among the various equations presented in \cite{rossJCAM-1} but without the statement of Theorem~\ref{thm:dynamics4opt}.  Hence, Theorem~\ref{thm:dynamics4opt} can also be proved using the first-principles approach of \cite{rossJCAM-1}.  However, because this first-principles' approach is lengthy, a shorter proof is presented here but one that relies on the theorems of Karush, John, Kuhn and Tucker\cite{bazaraa-2006,NW:NumOptBook,luenberger-2008}.
\end{remark}
%-------------------------
\begin{remark}
It is shown in \cite{rossJCAM-1} that the Pontryagin Hamiltonian for Problem~$(O_B)$ is given by $\langle (q_3(t) + \partial_{q_1}L(q_0(t), q_1(t), q_2(t))),\ u_1(t) \rangle $ and that the minimized Hamiltonian vanishes along an extremal.  As a result we get \eqref{eq:q3=bydef} with $\widetilde{q}_3(t) = q_3(t)$ for $u_1(t) \ne 0$.  In the current proof of Theorem~\ref{thm:dynamics4opt} we have simply chosen to define $\widetilde{q}_3(t)$ according to \eqref{eq:q3=bydef}. If $\widetilde{q}_3(t)$ and hence $q_3(t)$ is defined without the negative sign in \eqref{eq:q3=bydef}, Theorem~\ref{thm:dynamics4opt} holds but with a corresponding change in the sign of the $q_3$-dynamics in \eqref{eq:q-dynamics}. Choosing to define $q_3(t)$ with the negative sign simply maintains the consistency of Theorem~\ref{thm:dynamics4opt} with the first-principles' results of \cite{rossJCAM-1}.  Furthermore, the natural definition of $q_3(t)$ with the negative sign generates an asymmetric saddle point matrix when the resulting ideas are applied to address a constrained optimization problem\cite{rossJCAM-1}.  See also Section~\ref{sec:examples}.  The asymmetric saddle point matrix has (under appropriate conditions\cite{saddlePt-05,saddlePt-06,saddlePt-08})  better numerical properties than its symmetric counterpart.
\end{remark}
%======================
In view of Theorem~\ref{thm:dynamics4opt}, we now provide a more precise definition of a hidden algorithm primitive for Problem~$(N)$ that is consistent with the use of the symbol $q \in \real{2(1 +n + m)}$.
%
%--------------------
\begin{definition}[\emph{Hidden Algorithm Primitive for Problem~$(N)$}]\label{def:algor-primitive-P}
Let $\dot q(t) = F(q(t),u(t))$ be defined by \eqref{eq:q-dynamics}. Suppose there exists a (measurable) control function $t \mapsto u(t) \in \real{1+n+m}$ that drives a given a point $q(t_0) = q^{00} \in \real{2(1 +n + m)}$ to $q(t_f) \in T$, where $T$ is given by \eqref{eq:T-def}.  Then a (Carath\'{e}odory) solution $[t_0, t_f] \ni t \mapsto q(t)$ to the IVP,
\emph{}\begin{equation}\label{eq:algorithm-primitive-IVP}
    \dot q= F(q,u(t)), \quad q(t_0) = q^{00}
\end{equation}
is called a hidden algorithm primitive for Problem~$(N)$.
\end{definition}

\subsection{A Hidden-Algorithm-Primitive Formulation of Optimal Optimization }

Let $\U(q,t) \subset  \real{1+n+m} $ be a compact set of allowable values of $u$ that jointly depends on $q$ and $t$. Let $t \mapsto q(t)$ be a hidden algorithm primitive generated by a control function $t \mapsto u(t) \in \U(q(t), t)$ acting on $\dot q = F(q, u)$ from some given point $q(t_0) = q^{00}$.  Suppose the optimality of the resulting hidden algorithm primitive $q(\cdot)$ is determined by a cost functional defined by\cite{ross-book,vinter},
\begin{equation}
J[q(\cdot), u(\cdot), t_f] :=  \ell(q(t_f), t_f) + \int_{t_0}^{t_f} \mathcal{L}(q(t), u(t), t)\, dt
\end{equation}
where $\ell:(q, t) \mapsto \Real$ is a given differentiable endpoint (or Mayer) cost function and\linebreak $\mathcal{L}: (q, u, t) \mapsto \Real$ is a given differentiable running (or Lagrange) cost function. Then the optimal control problem that defines an optimal hidden algorithm primitive for Problem~$(N)$ is given by,
%
%---------------------------------------------------------
\begin{eqnarray}\label{eq:prob-R}
&(M) \left\{
\begin{array} {lll}
\displaystyle\mathop\text{minimize }_{[q(\cdot), u(\cdot), t_f]}  & J[q(\cdot), u(\cdot), t_f] \\
\text{subject to} & \dot q(t) = F(q(t), u(t))\\
                    & u(t) \in \U(q(t), t)\\
                    & q(t_0) = q^{00}\\
                    & q(t_f) \in T
\end{array} \right.&
\end{eqnarray}
%---------------------------------------------------------
%
\begin{remark}
At this stage, it is not particularly important to define a precise functional form for $\ell$ and/or $\mathcal{L}$ other than the implication that they are chosen by the problem designer.
\end{remark}
%=======================================
The following are standard definitions in optimal control theory\cite{clarke-green-book,vinter,ross-book}:

%-------------------------
\begin{definition}(Pontryagin Hamiltonian)\label{def:HP}
The Pontryagin Hamiltonian $H\left(q, \lambda, u, t \right)$ for Problem~$(M)$ is defined by\cite{clarke-green-book,vinter,ross-book},
\begin{equation}\label{eq:HP}
H\left(q, \lambda, u, t \right) := \mathcal{L}(q, u, t) + \left\langle \lambda, F(q, u) \right \rangle
\end{equation}
where, $\lambda \in \real{2(1+n+m)}$ is a Pontryagin adjoint covector.
\end{definition}
%-----------------
\begin{definition}[Lower Hamiltonian]
The lower Hamiltonian $\mathcal{H}\left(q, \lambda, t\right)$ for Problem~$(M)$ is defined by\cite{clarke-green-book,ross-book},
\begin{equation}\label{eq:H=minOfHP}
 \mathcal{H}\left(q, \lambda, t\right) := \min_{u \in \U(q, t)} H\left(q, \lambda, u, t \right)
\end{equation}
\end{definition}
%-------------------
\begin{definition}(Hamilton-Jacobi Equation)\label{def:HJE}
The Hamil\-ton-Jacobi equation for Problem~$(M)$ is the partial differential equation given by\cite{vinter,clarke-2013book},
\begin{equation}\label{eq:HJ}
\mathcal{H}\left(q, \partial_q V(q,t), t\right)  + \partial_t V(q,t) = 0
\end{equation}
where the function $V: (q, t) \mapsto \Real$ satisfies the boundary condition,
\begin{equation}
V(q(t_f), t_f) = \ell(q(t_f), t_f)
\end{equation}
over all values of $q(t_f) \in T$.
\end{definition}
%-------------------
%
\begin{theorem}[Optimal Optimization]\label{thm:verification}
Let $t \mapsto q^\flat(t)$ be a hidden algorithm primitive for Problem~$(N)$ generated by the action of a control function $t \mapsto u^\flat(t)$ from the point $q(t_0) = q^{00}$. Suppose there exists a continuously differentiable function $\varphi: (q,t) \mapsto \Real$ that satisfies
\begin{subequations}
\begin{align}
\mathcal{H}\left(q, \partial_q \varphi(q,t), t\right)  + \partial_t\varphi(q,t) &\ge 0 &\forall\ q, \forall\ t \in [t_0, t_f] \label{eq:HJI>0}\\
\varphi(q(t_f), t_f) &= \ell(q(t_f), t_f) & \forall q(t_f) \in T \\
\varphi(q(t_0), t_0) & = J[q^\flat(\cdot), u^\flat(\cdot), t^\flat_f]
\end{align}
\end{subequations}
Then $q^\flat(\cdot)$ is an optimal hidden algorithm primitive with optimal value given by $\varphi(q^{00}, t_0)$.
\end{theorem}
%-----------------------
\textsl{}A proof of this theorem follows directly from the well-known verification theorem in optimal control\cite{clarke-2013book,vinter}.

From Theorem~\ref{thm:verification}, it follows that a constructive approach to obtain an optimal control and hence an optimal hidden algorithm primitive is to solve the Hamilton-Jacobi equation and use the resulting control function to solve the IVP given by \eqref{eq:algorithm-primitive-IVP} as follows:
\begin{subequations}\label{eq:optHiddenAlgPrim}
\begin{align}
u(q,t) &= \arg \min_{u \in \U(q, t)} H\left(q, \partial_q V(q,t), u, t \right)\\
\dot q &= F(q, u(q,t)), \quad q(t_0) = q^{00}\label{eq:optHiddenAlgPrim-IVP}
\end{align}
\end{subequations}
However, as explained in Section~\ref{sec:intro}, our objective was not to construct an optimal hidden algorithm primitive in an explicit form; rather, our motivation was to seek out the fundamental mathematical physics of optimization.  In this context, Theorem~\ref{thm:verification} and \eqref{eq:optHiddenAlgPrim} accomplish this objective by establishing the fact that the natural physics of optimal optimization (via an optimal hidden algorithm primitive) is described by a Hamilton-Jacobi equation.

\subsection{Some Remarks On the Natural Physics of Optimization}
Although we have not yet fully connected a hidden algorithm primitive to an actual algorithm, some remarks on the developments of the results up to this point are in order. We begin by noting that Theorem~\ref{thm:dynamics4opt} creates a natural vector field $F(\cdot, u)$ over the space $2(1+n+m)$ using the $(1+m)$ data functions of Problem~$(N)$ and its $n$ optimization variables.  This vector field permeates a ``lifted'' space $\real{2(1+n+m)}$ and carries with it the information about the optimality of Problem~$(N)$.  A hidden algorithm primitive is a trajectory in $\real{2(1+n+m)}$. A standard guess to a solution of Problem~$(N)$\cite{luenberger-2008,NW:NumOptBook} is defined by,
\begin{equation}\label{eq:guess-01}
q_0(t_0) = 1, \quad q_1(t_0) = q_1^{00}, \quad q_2(t_0) = q_2^{00}
\end{equation}
where $q_1^{00}$ and $q_2^{00}$ are guesses of the optimization variable and the constraint multiplier respectively. Consequently, an initial point $q^{00} \in \real{2(1+n+m)}$ for Problem~$(M)$ and hence \eqref{eq:optHiddenAlgPrim-IVP} may be generated according to the following construction:
\begin{multline}\label{eq:guess-00}
q_0(t_0) = 1, \quad q_1(t_0) = q_1^{00}, \quad q_2(t_0) = q_2^{00}, \\
q_3(t_0) = - \partial_{q_1}L\left(1, q_1^{00}, q_2^{00}\right), \quad
q_4(t_0) = g\left(q_1^{00}\right), \quad
q_5(t_0) = g_0(q_1^{00})
\end{multline}
Any feasible control function $u(\cdot)$ drives the given initial point $q^{00}$ to a candidate optimal point $q(t_f) \in T$ by sensing the instantaneous values of the derivatives of the data functions of Problem~$(N)$ via the vector field $F(\cdot, u)$.  The minimization of the Pontryagin Hamiltonian is a local action that generates an instantaneous value of the control action predicated on the knowledge of the transversality condition. Thus, a continuous concatenation of the local control action delivers the global result of optimal optimization wherein the optimality criterion is specified by the functional $J[q(\cdot), u(\cdot), t_f]$.  A main theoretical challenge at this stage is to develop meaningful measures of optimality for optimal optimization.  For example, one could conceive the total variation of $q(\cdot)$ as a valid measure of optimality for a hidden algorithm primitive. Thus, a potentially viable path to advance the preceding ideas is to design desirable optimality criteria for optimal optimization. In the next section we embark on a different path toward optimality by way of inverse optimality\cite{glad87,freeman}.  It will be apparent shortly that this approach generates a number of near-term practical results.

\section{A Theory for Generating Inverse Optimal Algorithms}\label{sec:inv-opt-theory}

From Section~\ref{sec:HJ-theory}, it follows that the Hamilton-Jacobi framework constitutes a theoretical foundation for optimal hidden algorithm primitives. To migrate this idea to a practical algorithmic framework, we jointly employ the notion of inverse optimal control\cite{glad87,freeman} and proximal aiming\cite{prox-aiming-1994,clarke-2013book}.  It will be apparent shortly that these two ideas obviate the need to solve the Hamilton-Jacobi equation or directly produce a hidden algorithm primitive although both constructs are essential to the development of a theory for generating inverse-optimal algorithms.

\subsection{Inverse Optimal Hidden Algorithm Primitives}\label{sec:inv-opt-hid-prim}
In inverse optimal control, one flips the roles of the functional $J[q(\cdot), u(\cdot), t_f]$ and the value function $V(q,t)$ in \eqref{eq:HJ}.  That is, instead of solving for $V(q,t)$ for a given functional $J[q(\cdot), u(\cdot), t_f]$, one selects $V(q,t)$ a priori and determines $J[q(\cdot), u(\cdot), t_f]$ afterwards.  This role reversal is made meaningful by the notion that a selection of $V(q,t)$ parameterizes a family of control functions $u(\cdot)$ that indirectly minimizes some cost functional $J[q(\cdot), u(\cdot), t_f]$\cite{haddad-2015,freeman}.   Furthermore, because the task of a control in \eqref{eq:q-dynamics} is to always drive any given point $q(t_0) = q^{00}$ to $q(t_f) \in T$, a hidden algorithm primitive can be reframed in terms of stabilizability (with respect to $T$) of the dynamical system $(D)$. Under this concept, $V(q,t)$ takes the role of a control Lyapunov function\cite{freeman,clarke-2013book,sontag-book}; i.e., a generalization of the classical Lyapunov function.   Thus the link between stabilizability and optimality is that a control Lyapunov function is a meaningful value function\cite{freeman}.

\subsubsection{Partial Stabilizability and Guidability}

The stabilizability requirements for $(D)$ are substantially weaker than those frequently encountered in control theory. For instance, we are only interested in partial stabilizability\cite{vorotnikov-2002,haddad-2015,chellabonia-2002,jammazi-2014,zuev-2000} because the target set $T$ does not directly involve $q_1$ or $q_5$.  In view of this particular requirement, we split $(D)$ into two dynamical systems,
%-------------------------------------
\begin{subequations}\label{eq:D=A+B}
\begin{align}
&(A) \left\{
\begin{aligned}
\dot q_0(t) &=u_0(t)\\
\dot q_2(t) & = u_2(t)\\
\dot q_3(t)  &= -\big[\partial^2_{q_1}L(q_0(t), q_1(t), q_2(t))\big] \, u_1(t) - \partial_{q_1}L(u_0(t), q_1(t), u_2(t))\\
\dot q_4(t) &=  \big[\partial_{q_1}g(q_1(t))\big] u_1(t)
\end{aligned}
\right. \label{eq:A-def-0}\\
%----
&(B) \left\{
\begin{aligned}
\dot q_1(t) & = u_1(t) \\
\dot q_5(t) & = \langle \partial_{q_1} g_0(q_1(t)), u_1(t) \rangle
\end{aligned}
\right.
\end{align}
%---------
\end{subequations}
%=======================================
%
where only $(A)$ needs to be stabilizable with respect to $T$ while the trajectories of $(B)$ must be merely bounded.  In other words, the optimization variable $q_1$ must be bounded and the cost function $q_5$ must be bounded from below, a standard assumption in optimization\cite{luenberger-2008,bazaraa-2006,NW:NumOptBook}.  Furthermore, the weak requirements of partial stabilizability of $(D)$ are even weaker for optimization than typical cases considered in control theory\cite{vorotnikov-2002,haddad-2015,zuev-2000}.  This is because, for the purposes of optimization, we do not seek partial stabilizability for disturbance rejection or robust design or other common criterion used in control theory\cite{sontag-book,clarkeLyap,freeman}. Hence, we do not require $u(\cdot)$ to be smooth or be explicitly defined in terms of a feedback law, $u = K(q,t)$. In principle, an open-loop control suffices. All of these weak requirements imply Clarke's notion of guidability\cite{clarkeLyap} suffices in terms of generating hidden algorithm primitives. That is, from the point of view of solving Problem~$(N)$ via the dynamical system $(D)$, we only need to partially guide the state $q$ to $T$ rather than keep it close to $T$ when the initial condition is already near the target.  The latter is the stability issue while the former is the (partial) guidability problem\cite{clarkeLyap}.

\subsubsection{Search Lyapunov Functions}

In view of all of the weak requirements discussed in the preceding paragraph, we define a search Lyapunov function (SLF) $S :(q,t) \mapsto \Real_+$ as a control Lyapunov function associated with the partial guidability to $T$ of the dynamical system $(D)$ (i.e, guidability of the dynamical system $(A)$ with respect to the set $T$).  Hence, following the results for Lyapunov functions presented in \cite{haddad-2015,zuev-2000,chellabonia-2002,clarkeLyap,clarke-2013book}, we define an SLF as follows:
%
%------------------
\begin{definition}\label{def=SLF}
A continuous function $S:(q,t) \to \Real_+$ is called an SLF for the pair $\big((D), T\big)$ if there exists a continuous rate function $R:(q,t) \to \Real_+$ and a class $\mathcal{K}$ function\cite{vorotnikov-2002,chellabonia-2002} $\beta:~\Real_+ \to \Real_+$ with respect to $T$ such that the following conditions of positive semidefiniteness, growth and infinitesimal decrease hold:
\begin{subequations}\label{eq:SLF-defx3}
\begin{align}
& S(q, t) = 0  \iff \ q \in T, \quad R(q,t) =0 \iff q \in T \label{eq:S=0@opt}\\
&S(q, t) \ge \beta(\|q_a\|) \quad \forall\ q \not\in T, \ \forall\ t \ge t_0\\
&\exists u \in \U(q,t) \ \text{such that }  DS(q, t; F(q, u)) \le - R(q,t) \quad \forall\ q \not\in T, \ \forall\ t \ge t_0\label{eq:diss-0}
\end{align}
\end{subequations}
where $q_a := (q_0, q_2, q_3,q_4)$,  $\beta (\cdot)$ is a continuously increasing unbounded function satisfying the condition $\beta(\|q_a\|) = 0$ for all $ q_a \in T$ and $DS(q, t; F(q, u))$ is a generalized directional derivative of $S(q,t)$ along the direction $\big(F(q,t), 1\big)$.
\end{definition}
%-------------------------
For a continuously differentiable SLF, $DS(q, t; F(q, u))$ is given by,
\begin{equation}\label{eq:DS:=}
D S(q, t; F(q,u)) := \langle \partial_q S(q,t), F(q,u) \rangle + \partial_t S(q, t)
\end{equation}
If the SLF is not continuously differentiable, then $D S(q, t; F(q,u))$ in \eqref{eq:diss-0} is an appropriately generalized directional derivative based on the regularity of the SLF\cite{clarke-2013book,clarke-nolcos-2010,osinenko-2020}. Of these generalized directional derivatives, the Dini derivate\cite{clarke-2013book,osinenko-2020} turns out to be one of the most useful.  See Section~\ref{sec:nonsmooth-stuff} for further discussions.
%-------------
\begin{remark}
Not all requirements of an SLF indicated in \eqref{eq:SLF-defx3} are needed in all situations.  For instance,  for a continuously differentiable SLF, the rate constraint in \eqref{eq:diss-0} is not necessary\cite{clarke-nolcos-2010,chellabonia-2002}.  However, if the Bhat-Bernstein rate constraint\cite{haddad-2015,bhat-2000,polyakov-2014} is imposed on a continuously differentiable SLF, then convergence can be achieved in finite time, which may be particularly useful for optimization\cite{ross:accelerated-arxiv}.  On the other hand, additional requirements on the SLF must be imposed for other types of partial- stability/guidability requirements such as partial global asymptotic guidability, partial uniform Lyapunov stability etc.\cite{chellabonia-2002, vorotnikov-2002, zuev-2000, osinenko-2020, clarke-nolcos-2010, jammazi-2014, haddad-2015}.  We deliberately avoid discussing the myriad notions of partial- stability/guidability and the corresponding requirements they impose on the SLF in order to focus on using the most basic set of ideas that are necessary for the exclusive purposes of optimization.
\end{remark}
%--------------

Substituting \eqref{eq:DS:=} in \eqref{eq:diss-0} (and, hence, disregarding the rate constraint) we get the other side of the Hamilton-Jacobi inequality (Cf.~\eqref{eq:HJI>0}),
\begin{equation}\label{eq:HJI<0}
\partial_t S(q, t) + \min_{u \in \U(q,t)}\langle \partial_q S(q,t), F(q,u) \rangle < 0\quad \forall\ q \not\in T
\end{equation}
Comparing \eqref{eq:HJI<0} with \eqref{eq:HP}, \eqref{eq:H=minOfHP} and \eqref{eq:HJ} the inverse-optimality argument follows from the existence of some running cost $\mathcal{L}(q(t),u(t),t) > 0$ that is minimized by the choice of $S(q,t)$\cite{freeman,haddad-2015}.

\subsubsection{Process for Generating an Inverse-Optimal Hidden Algorithm Primitive}

Collecting all of the key ideas, the generation of an inverse-optimal hidden algorithm primitive can be summarized as follows:
\begin{enumerate}
\item Given an optimization Problem~$(N)$ generate the dynamical system $(D)$ (and hence $(A)$ and $(B)$).  This step is straightforward and only requires the problem data functions and their derivatives.
\item Select a search Lyapunov function $S: (q, t) \mapsto \Real_+$.
\item Choose $\U(q,t)$, a compact set of allowable values for $u$.
\item Solve \eqref{eq:HJI<0} for $u(q, t)$.
\item Solve the IVP given by \eqref{eq:optHiddenAlgPrim-IVP}.
\end{enumerate}
\begin{remark}
The selection of $S(q,t)$ and $\U(q,t)$ in Steps~2 and 3 may be interdependent because of the requirement stipulated by \eqref{eq:diss-0}.
\end{remark}
%==================

\subsection{Inverse Optimal Algorithm Generators}\label{sec:inv-opt-algol}
From Section~\ref{sec:inv-opt-hid-prim}, it follows that the notion of inverse optimality allows us to escape the problem of solving the Hamilton-Jacobi equation for optimal optimization (via a hidden algorithm primitive).  We now use a procedure inspired by proximal aiming\cite{prox-aiming-1994,clarke-2013book} and practical stabilizability\cite{prac-stab-1980,prac-stab-1985,clarkeLyap} to escape the problem of generating a hidden algorithm primitive as a means to produce an algorithm for solving Problem~$(N)$.  We briefly note that this escape hatch still requires the notion of a hidden algorithm primitive.

\subsubsection{Practical Guidability and Polygonal Arcs}
From the point of view of convergence of an optimization algorithm, it is not necessary for $q(t_f)$ to solve a given problem exactly.  That is, it suffices for $q(t_f)$ to be within an arbitrarily specified $\epsilon$-ball around $T$. In fact, in most practical cases implemented on a digital computer, the value of $\epsilon > 0$ cannot be smaller than $\sqrt{\epsilon_M}$, where $\epsilon_M$ is the machine precision. Regardless, the concept of $q(t)$ approaching $T$ within an $\epsilon > 0$ value is known in control theory as ``practical'' stabilization\cite{prac-stab-1980, prac-stab-1985, clarke-nolcos-2010, osinenko-2020}.  Obviously, this weaker concept of stabilizability provides the right fit for an optimization algorithm versus other stronger notions such as continuous asymptotic controllability.

One approach used in control theory for practical guidability is to produce a feedback control $u = K(q, t)$ in the sample-and-hold sense\cite{clarke-nolcos-2010,osinenko-2020}. That is, a control is sampled at time $t_k$ and held a constant value $K(q_k, t_k), \ q_k = q(t_k)$ up to time $t_{k+1} > t_k$.  In between the sample points, $[t_{k+1}, t_k] \ni t \mapsto q(t)$ is a solution to the differential equation ${\dot q}(t) = F(q(t), K(q_k, t_k))$ with initial condition $q(t_k) = q_k$.  A direct application of this control-theoretic idea for the purposes of optimization would produce an inverse-optimal hidden algorithm primitive in the sample-and-hold sense.  However,\, \emph{from an algorithmic perspective, the value of $q(t)$ in between the sample points is totally irrelevant}.
In fact, even the precise location of the points, $q(t_k)$, are also  irrelevant (except at $t = t_f$). This is simply because the goal of a convergent algorithm is to merely arrive at a point in $T$.  Hence, we are not interested in a point-wise accurate, discrete-time solution to ${\dot q} = F(q, K(q, t))$\cite{boggs71,brown+biggs}. Because of this unique requirement of an optimization algorithm, conventional discretization methods for ODEs\cite{hnw-ode} are not necessarily relevant to the production of efficient algorithms.  The only requirement for a convergent algorithm generator is to produce a sequence $q(t_0), q(t_1), \ldots $ that accumulates near $T$ within an $\epsilon$ tolerance with no burden to track any solution of the flow resulting from ${\dot q} = F(q, K(q, t))$.  Hence, if we were to somehow jump from $q(t_k)$ to $q(t_{k+1})$, $k=0, 1, \ldots$, the resulting polygonal arc would directly produce an (inverse-optimal) algorithm without ever passing through the process of generating a hidden algorithm primitive. We achieve these controlled jumps by large discrete decrements of the SLF.

\subsubsection{Discrete Decrements of an SLF}
An inverse-optimal algorithm uses an SLF to produce a sequence of strictly decreasing numbers,
\begin{equation}\label{eq:SLF-seq}
S(q(t_0), t_0) > S(q(t_1), t_1) > \cdots S(q(t_{k}), t_{k}) > S(q(t_{k+1}), t_{k+1}) \cdots
\end{equation}
so that (Cf.~\eqref{eq:SLF-defx3})
\begin{equation}\label{eq:Slim=0}
\lim_{k \to \infty} S(q(t_k), t_k) = 0
\end{equation}
Hence, the limit of the sequence $q(t_0), q(t_1), \cdots $ resulting from \eqref{eq:SLF-seq} is (Cf.~\eqref{eq:S=0@opt}) a solution to Problem~$(N)$:
\begin{equation}\label{eq:qlim=sol}
\lim_{k \to \infty} q(t_k) := q_\infty \in T
\end{equation}
%
%---------------------
\begin{lemma}\label{lemma:Euler-arc}
Let $S: (q, t) \to \Real_+$ be a continuously differentiable SLF. Let $u(t_k) = u$ be a solution to the problem,
\begin{eqnarray}\label{eq:Step-0}
&(M_u)\left\{
\begin{array}{lrl}
\displaystyle\mathop{\text{minimize }}_u && DS(q(t_k), t_k; F(q(t_k), u))  \\
\text{subject to}
&& u \in \U(q(t_k),t_k)
\end{array} \right.&
\end{eqnarray}
at any point $(q(t_k), t_k)$. Assume $q(t_k) \not\in T$.  Then there exists a sequence $t_k, \ k= 0, 1, \ldots$, $t_{k+1} > t_k$ and a polygonal arc given by,
\begin{equation}\label{eq:q-poly-arc}
q(t_{k+1}) = q(t_k) + F(q(t_k), u(t_k))(t_{k+1} -t_k)
\end{equation}
such that \eqref{eq:qlim=sol} holds.
\end{lemma}
%--------------------
\begin{proof}
By definition of an SLF, at any $t=t_k$, there exists a sufficiently small time interval $\delta t > 0$ such that,
\begin{equation}\label{eq:delta-t-argument}
S(q(t+ \delta t), t+ \delta t) - S(q(t), t) = DS\big(q(t), t; F(q(t), u(t_k))\big)\,\delta t < 0
\end{equation}
Setting $t+\delta t$ to $ t_{k+1}$ in \eqref{eq:delta-t-argument} we get,
\begin{equation}\label{eq:S-disc-dec-proof}
S(q(t_{k+1}), t_{k+1}) < S(q(t_{k}), t_{k}) \quad k = 0, 1, \ldots
\end{equation}
Similarly from Theorem~\ref{thm:dynamics4opt} it follows that there exists a sufficiently small interval $\delta t >0$ around $t=t_k$ such that,
\begin{equation}\label{eq:q-prop4deltat}
q(t+\delta t) - q(t) = F(q(t), u(t_k)) \delta t
\end{equation}
Setting $t+\delta t$ to $ t_{k+1}$ in \eqref{eq:q-prop4deltat} we get \eqref{eq:q-poly-arc}.  Since $u(t_k)$ decrements the SLF at each $t_k$ we get \eqref{eq:Slim=0} from the property of an SLF. Hence, \eqref{eq:qlim=sol} follows.
\end{proof}
%=====================================
\begin{lemma}\label{lemma:max-h}
Let the assumptions of Lemma~\ref{lemma:Euler-arc} hold. Let $u(k)$ be a solution to Problem~$(M_u)$. Let $h_k = h$ be a solution to the following problem:
\begin{eqnarray}\label{eq:exact-line-search}
&(M_h)\left\{
\begin{array}{lll}
\displaystyle\mathop{\text{minimize }}_{(h > 0,\ q(k+1))} & S(q(k+1), t_k + h)  \\
\text{subject to}
& q_0(k+1)- q_0(k) - h\ u_0(k) = 0\\
&q_1(k+1) - q_1(k) - h\ u_1(k) = 0\\
&q_2(k+1) - q_2(k) - h\ u_2(k) = 0 \\
&q_3(k+1) + \partial_{q_1}L(q_0(k+1), q_1(k+1), q_2(k+1)) = 0\\
&q_4(k+1) -  g(q_1(k+1)) = 0 \\
&q_5(k+1) - g_0(q_1(k+1)) = 0
\end{array} \right.&
\end{eqnarray}
Then $S(q(k+1), t_k+h_k)$ is maximally decremented in the direction $DS\big(q(t), t; F(q(t), u(k))\big)$.
\end{lemma}
%-------------------
\begin{proof}
From the proof of Theorem~\ref{thm:dynamics4opt}, it is obvious that the dynamical system $(D)$ has the following three integrals of motion:
\begin{subequations}\label{eq:integrals-of-motion}
\begin{align}
q_3(t) + \partial_{q_1}L(q_0(t), q_1(t), q_2(t)) &=0\\
q_4(t) - g(q_1(t)) &=0\\
q_5(t) - g_0(q_1(t)) &=0
\end{align}
\end{subequations}
Hence, all evolutions of $\dot q = F(q,u)$ lie on the hypersurface defined by \eqref{eq:integrals-of-motion}. Setting $t = t_k +h$ in \eqref{eq:integrals-of-motion}, it follows that any point $q(k+1)$ that satisfies the equation,
\begin{subequations}\label{eq:jump-proof}
\begin{align}
q_3(k+1) + \partial_{q_1}L(q_0(k+1), q_1(k+1), q_2(k+1)) &= 0\\
q_4(k+1) - g(q_1(k+1)) &=0 \\
q_5(k+1) - g_0(q_1(k+1)) &=0
\end{align}
\end{subequations}
is on the hypersurface defined by \eqref{eq:integrals-of-motion}. These hypersurface constraints constitute the last three constraints that define Problem~$(M_h)$. Define,
\begin{subequations}\label{eq:Euler-jump}
\begin{align}
q_0(k+1) &:= q_0(k) + h\ u_0(k)\\
q_1(k+1) &:= q_1(k) + h\ u_1(k) \\
q_2(k+1) &:= q_2(k) + h\ u_2(k)
\end{align}
\end{subequations}
Then for sufficiently small values of $h >0 $ in \eqref{eq:Euler-jump}, $q(k+1)$ satisfies the conditions for Lemma~\ref{lemma:Euler-arc} to hold.  As a result, we have $S(q(k+1), t_k + h) < S(q(k), t_{k}),\ k = 0, 1, \ldots$ for a sufficiently small value of $h>0$ in Problem~$(M_h)$. Furthermore, for any $h>0$ we have  $S(q(k+1), t_k + h) \ge 0$ (by definition).  Hence, a solution to Problem~$(M_h)$ produces a value of $h = h_k$ that generates the maximum possible decrement of $S(q(k+1), t_k + h_k)$ in the direction $DS\big(q(t), t; F(q(t), u(k))\big)$.
\end{proof}
%============================
\begin{remark}
From Lemma~\ref{lemma:max-h}, it follows that the sequence $q(0), q(1), \ldots $ is a polygonal arc that does not necessarily track a hidden algorithm primitive. Furthermore, it is apparent that $h$ generated by solving Problem~$(M_h)$ may not satisfy any of the established consistency conditions for discretizations of differential equations\cite{hnw-ode}.
\end{remark}
%----------------------
%
%-----------------
\begin{remark}\label{remark:exactLSbutNot}
At first glance, Problem~$(M_h)$ appears similar to an exact line search algorithm\cite{luenberger-2008,bazaraa-2006,NW:NumOptBook}; however, there are three sharp differences between the two:
\begin{enumerate}
\item Problem~$(M_h)$ is a constrained optimization problem;
\item The objective function in Problem~$(M_h)$ is an SLF; and
\item The global minimum of an SLF is zero (Cf.~\eqref{eq:S=0@opt}).
\end{enumerate}
\end{remark}
%------------------

\subsubsection{An Inverse Optimal Algorithm Generator}
An inverse optimal algorithm essentially comprises two steps.  In the first step, Problem~$(M_u)$ is solved to produce $u =u(k)$.  This value of $u(k)$ is used in second step to generate $h = h_k$ and thus the next iterate $q(k+1)$.
\begin{theorem}\label{thm:inv-opt-algol}
Let $u(k)$ and  $h_k$ be solutions to Problems~$(M_u)$ and $(M_h)$ respectively. Let $q(0)$ be given by $q(t_0)$ defined in \eqref{eq:guess-00}.  Assume $q(0) \not\in T$.  Let $q(k+1)$ be given by the iterative map,
\begin{equation}\label{eq:iterativeMap}
\begin{aligned}
q_0(k+1) &\leftarrow\ q_0(k) + h_k\ u_0(k)\\
q_1(k+1)  &\leftarrow\ q_1(k) + h_k\ u_1(k) \\
q_2(k+1) &\leftarrow\ q_2(k) + h_k\ u_2(k) \\
q_3(k+1) &\leftarrow\ -\partial_{q_1}L(q_0(k+1), q_1(k+1), q_2(k+1))\\
q_4(k+1) &\leftarrow\ g(q_1(k+1)) \\
q_5(k+1) &\leftarrow\ g_0(q_1(k+1))
\end{aligned}
\end{equation}
%=====
for $k = 0, 1, \ldots$. Then $\lim_{k \to \infty} q(k)$ converges to a solution of Problem~$(N)$.
\end{theorem}
%------------------
\begin{proof}
The proof of this theorem follows directly from Lemmas~\ref{lemma:Euler-arc} and \ref{lemma:max-h}.
\end{proof}
%===================

\subsubsection{A Practical Inverse Optimal Algorithm Generator}\label{sec:practialInvOptAlgol}
The objective function in Problem~$(M_u)$ is linear in $u$.  Hence, the ease (or difficulty) in solving Problem~$(M_u)$ is primarily dictated by the choice of $\U(q,t)$.  For example, if $\U(q,t)$ is convex for all $(q,t)$ then Problem~$(M_u)$ is a convex optimization problem.  Hence a practical inverse optimal algorithm is predicated on the fact that Problem~$(M_u)$ is an easier problem to solve than Problem~$(N)$.

Because $q(k+1)$ can be computed explicitly from $q(k)$, Problem~$(M_h)$ may be considered to be a univariate minimization problem.  Furthermore, because the global minimum of an SLF is zero, one possible choice for a guess of $h = h^0$ to initiate a solution to Problem~$(M_h)$ is given by\cite{ross:accelerated-arxiv},
\begin{equation}\label{eq:alpha-guess-4SLF}
h^0 = \frac{S(q(k), t_k)}{-DS(q(k), t_k; F(q(k), u(k))}
\end{equation}
For much the same reason as why practical optimization algorithms do not solve the exact line search problem, a practical inverse optimal algorithm may solve Problem~$(M_h)$ approximately. In this case, \eqref{eq:alpha-guess-4SLF} may also be used as a starting point to solve Problem~$(M_h)$ approximately and backtrack when necessary\cite{ross:accelerated-arxiv}.  Furthermore, because the three data points, $S(q(k), t_k)$, $DS(q(k), t_k; F(q(k), u(k))$ and $S(q(k+1), t_k+h^0)$ are available, backtracking may also be performed via minimizing Hermite interpolations of $t \mapsto S(q(t), t)$. In other words, similar to a practical line-search algorithm, one may choose to solve Problem~$(M_h)$ that is consistent with the idea that $S(q(k+1), t_k+h)$ has decreased sufficiently. Note also that an SLF does not necessarily prevent the cost function $g_0(q_1 (k+1))$ from increasing in value; hence, Problem~$(M_h)$ naturally allows ``hill climbing'' with respect to the cost function (provided the value of the SLF decreases at the next iterate).

Collecting all of the preceding ideas, the key steps of a practical inverse optimal algorithm generator can be summarized as follows:
%=====================================
\begin{enumerate}
\item[Step 0]
    \begin{enumerate}
 %   \item[]
    \item Using the data functions of a given Problem~$(N)$, generate its dynamical system~$(D)$ (Cf.~Theorem~\ref{thm:dynamics4opt}).
     \item Select an SLF $S(q,t)$ together with a compact control set $\U(q,t)$ such that Problem~$(M_u)$ is easy to solve (i.e., easier than Problem~$(N)$).
      \item Choose a small value for $\epsilon > 0$. Set $k=0$. Use \eqref{eq:guess-00} to evaluate $q(k) = q(t_0)$.
     \end{enumerate}
 \item[Step 1]If $S(q(k), t_k) \le \epsilon$  stop and exit.  Else, go to Step~2.
 \item[Step 2] Solve the (easy) Problem~$(M_u)$ to produce $u(k)$.
 \item[Step 3] Using $u(k)$ solve Problem~$(M_h)$ approximately to produce $h_k$.
 \item[Step 4] Advance to the next iterate using \eqref{eq:iterativeMap}. Set $k \leftarrow k+1$ and go to Step~1.
\end{enumerate}
%
%%====================
\begin{remark}
The stopping criterion in Step~1 of the inverse optimal algorithm uses the fact that the SLF vanishes at a candidate optimal point (Cf.~\eqref{eq:S=0@opt}).  From this perspective, an SLF may also be used as a measure of ``distance'' to optimality. Obviously, these ideas are totally different from the concept of Bregman divergence and the ensuing results developed in \cite{ODEinML-2021-1,ODEinML-2016-2,ODEinML-2021-3}.
\end{remark}
%-----------------
\begin{remark}
From the totality of steps outlined in its production, it is clear that a hidden algorithm primitive is never generated although its concept was crucial to the development of a practical inverse optimal algorithm.
\end{remark}
%------------------------
%
In summary, a practical inverse optimal algorithm generator uses a blend of old and new ideas with the transversality mapping principle providing the spark for an investigation of optimization via optimal control theory.

\subsection{Two Basic Minimum Principles for Generating Algorithms}\label{sec:MinPx2}
To limit the scope of the remainder of this paper, we now consider a narrower class of SLFs that do not depend on all of $q$ and $t$. These types of SLFs can be designed by recalling that only the variables of the dynamical system $(A)$  need to be guided to $T$.  In view of this observation, we simplify the notation for $(A)$ and rewrite it as,
\begin{equation}
\dot q_a = f(q_0, q_1, q_2, u)
\end{equation}
where $q_a:=(q_0, q_2, q_3, q_4)$ as before and $f(\cdot)$ is the right-hand-side of \eqref{eq:A-def-0}. We now define an SLF that depends only on $q_a$:
\begin{subequations}
\begin{align}
& S(q_a) = 0  \iff \ q_a \in T \\
&S(q_a) > 0 \quad \forall\ q_a \not\in T\\
&\exists u \in \U(q,t) \ \text{such that } \langle \partial_{q_a}S(q_a), f(q_0, q_1, q_2, u) \rangle < 0\quad \forall\ q_a \not\in T, \quad \forall\ t \ge t_0
\end{align}
\end{subequations}
Under this assumption Problem~$(M_u)$ simplifies to solving the linear-cost minimization problem,
%
%---------------------------------------------------------
\begin{eqnarray}\label{eq:minP}
&(P) \left\{
\begin{array} {lll}
\displaystyle\mathop\text{minimize }_{u \in \real{1+n+m}}  &\pounds_fS:= \langle \partial_{q_a}S(q_a), f(q_0, q_1, q_2, u) \rangle \\
\text{subject to} & u \in \U(q,t)
\end{array} \right.&
\end{eqnarray}
%---------------------------------------------------------
%
where the symbol $\pounds_fS$ is used to denote the fact that $\langle \partial_{q_a}S(q_a), f(q_0, q_1, q_2, u) \rangle$ is the Lie derivative of $S$ along the vector field $f$.

As noted earlier, in certain cases, it might be more useful to specify a minimum rate of descent for $\pounds_fS$.  Let $R : (q,t) \mapsto \Real_+$ be such a function that satisfies $R(q,t) = 0 \Leftrightarrow q_a \in T$.  This implies we require $u$ to satisfy the constraint,
\begin{equation}\label{eq:Sdot+R<0}
\pounds_f S + R(q,t) \le 0
\end{equation}
Let $Q: (u, q, t) \mapsto \Real $ be some objective function such that a solution to the following problem exists,
%
%---------------------------------------------------------
\begin{eqnarray}\label{eq:minP*}
&(P^*) \left\{
\begin{array} {lll}
\displaystyle\mathop\text{minimize }_{u \in \real{1+n+m}}  &Q(u, q, t) \\
\text{subject to} & \pounds_f S + R(q,t) \le 0
\end{array} \right.&
\end{eqnarray}
%---------------------------------------------------------
%
Then Problem~$(P^*)$ may be considered as an alternative to Problem~$(P)$ to determine $u$\cite{ross:accelerated-arxiv}.  Note that the constraint in Problem~$(P^*)$ is linear and merely one-dimensional.

In the next section, we use the minimum principles given by $(P)$ and $(P^*)$ to generate algorithms for an illustrative class of problems.

\section{Illustrative Examples}\label{sec:examples}
Consider the fundamental problem of solving an equality-constrained optimization problem given by,
%
%---------------------------------------------------------
\begin{eqnarray}\label{eq:prob-N_0}
&(N_0) \left\{
\begin{array} {lll}
\displaystyle\mathop\text{minimize }_{q_1 \in \real{n}}  & g_0(q_1) \\
\text{subject to} & g(q_1) = 0
\end{array} \right.&
\end{eqnarray}
%---------------------------------------------------------
%
where $g(q_1):= (g_1(q_1), \ldots, g_{m}(q_1))$ as before.  The standard Lagrangian for this problem is given by,
\begin{equation}
L(q_1, q_2) :=  g_0(q_1) +  \left\langle q_2, g(q_1) \right\rangle
\end{equation}
Suppose we set $q_0(t) = 1$ for all $t$ with $u_0(t) = 0\ \forall t$.  Then the target set $T = T_0$ for Problem~$(N_0)$ simplifies to (Cf.~\eqref{eq:T-def}),
\begin{eqnarray}
T_0:= \left\{q:\ \  q_3 = 0,\ \ q_4 = 0 \right\}\label{eq:T0-def}
\end{eqnarray}
Because there are no final-time conditions on $q_2$, the $q_2$-dynamics belong to the $(B)$ system (Cf.~\eqref{eq:D=A+B}); hence, we split the dynamical system $(D)$ (for Problem~$(N_0)$) according to,
\begin{subequations}
\begin{align}
&(A_0) \left\{
\begin{aligned}
\dot q_3(t)  &= -\big[\partial^2_{q_1}L(q_1(t), q_2(t))\big] \, u_1(t) - \big[ \partial_{q_1}g(q_1(t))\big]^T u_2(t)\\
\dot q_4(t) &=  \big[\partial_{q_1}g(q_1(t))\big] u_1(t)
\end{aligned}
\right. \label{eq:A_0-def}\\
%----
&(B_0) \left\{
\begin{aligned}
\dot q_1(t) & = u_1(t) \\
\dot q_2(t) & = u_2(t)\\
\dot q_5(t) & = \langle \partial_{q_1} g_0(q_1(t)), u_1(t) \rangle
\end{aligned}
\right.\label{eq:B_0-def}
\end{align}
%---------
\end{subequations}
and define $q_a:=(q_3, q_4)$ with $f(\cdot)$ given by the right-hand-side of \eqref{eq:A_0-def}.
From Section~\ref{sec:MinPx2}, it follows that a simple and tentative choice of an SLF for the partial guidability of the pair $(A_0, B_0)$ is the quadratic given by,
\begin{equation}\label{eq:SLF=quad-quad}
S(q_3, q_4) := \frac{1}{2}\langle q_3, q_3\rangle + \frac{1}{2}\langle q_4, q_4\rangle
\end{equation}
Application of the minimum principle $(P)$ given by \eqref{eq:minP} generates the linear-cost minimization problem,
%
%---------------------------------------------------------
\begin{eqnarray}\label{eq:minP_0}
&(P_0) \left\{
\begin{array} {cll}
\displaystyle\mathop\text{minimize }_{(u_1, u_2) \in \real{n} \times \real{m}}  &
- \langle q_3, \big[\partial^2_{q_1}L(q_1, q_2)\big] \, u_1 + \big[ \partial_{q_1}g(q_1)\big]^T u_2 \rangle +\langle q_4, \big[\partial_{q_1}g(q_1)\big] u_1 \rangle
\\
\text{subject to} & (u_1, u_2) \in \U(q,t)
\end{array} \right.&
\end{eqnarray}
%---------------------------------------------------------
%
Next, consider the family of quadratic sets for $\U(q,t)$ given by,
\begin{equation}\label{eq:U=quad}
\U(q,t) = \left\{(u_1, u_2):\ [u_1, u_2]^T W(q,t) \left[
                                               \begin{array}{c}
                                                 u_1 \\
                                                 u_2 \\
                                               \end{array}
                                             \right] \le \Delta(q,t)
   \right\}
\end{equation}
where $W(q,t)$ is an  $(n+m)\times (n+m)$ symmetric positive definite matrix that metricizes the search (control) space and $\Delta(q,t) > 0$ is some finite number that may depend on $q$ and/or $t$.
%
%--------------------
\begin{remark}
The quantity $\Delta(q,t)$ in \eqref{eq:U=quad} is similar to a trust region parameter if $W(q,t)$ is taken to be the identity matrix.  Nonetheless, it is still different from a trust region because the optimization problem $(P_0)$ has no solution if $u$ is strictly inside $\U(q,t)$.
\end{remark}
%----------------------
%
Solving Problem~$(P_0)$ with $\U(q,t)$ given by \eqref{eq:U=quad} yields,
\begin{equation}\label{eq:minP-quad-quad-result-0}
W(q,t) \left[
         \begin{array}{c}
           u_1 \\
           u_2 \\
         \end{array}
       \right] = \sigma \left[
                          \begin{array}{cc}
                           \partial^2_{q_1}L(q_1, q_2) & -\big[\partial_{q_1}g(q_1)\big]^T \\[0.5em]
                            \partial_{q_1}g(q_1) &  0 \\
                          \end{array}
                        \right] \left[
         \begin{array}{c}
           q_3 \\
           q_4 \\
         \end{array}
       \right], \qquad \sigma > 0
\end{equation}
where $(1/\sigma)$ is a multiplier associated with the control constraint. Deferring a discussion on the selection of $W(q,t)$, a particular class of inverse optimal algorithms for solving Problem~$(N_0)$ can be defined as follows:
\begin{enumerate}
\item Using \eqref{eq:guess-00} evaluate the SLF given by \eqref{eq:SLF=quad-quad}. If this value of the SLF is within a specified $\epsilon > 0 $ tolerance, stop and exit. Else, set the counter $k=0$ and continue.
\item Solve Problem~$(P_0)$ to produce $u(k) = (u_1, u_2)$. (Cf.~\eqref{eq:minP-quad-quad-result-0}.)
\item Solve Problem~$(M_h)$ (approximately) to generate $h_k$. (See also \eqref{eq:alpha-guess-4SLF}.)
\item Evaluate \eqref{eq:SLF=quad-quad} at $q(k+1)$ using \eqref{eq:iterativeMap}.  If $S(q_3(k+1), q_4(k+1)) \le \epsilon$ stop and exit.  Otherwise set $k \leftarrow k+1$ and go to Step 2.
\end{enumerate}
It is apparent from this procedure that a differential equation is not produced and solved to generate an inverse-optimal algorithm.

\subsection{Choosing $W(q,t)$, Part I}
The block $2\times 2$ matrix on the right-hand side of \eqref{eq:minP-quad-quad-result-0} is the asymmetric KKT matrix\cite{NW:NumOptBook,saddlePt-05,saddlePt-08}.  This equation can be easily transformed to produce the symmetric KKT matrix by changing the sign of the $q_3$ coordinate,
\begin{equation}\label{eq:qbar:=}
\overline{q}_3 := -q_3, \quad  \overline{q} := (q_1, q_2, \overline{q}_3, q_4, q_5)
\end{equation}
Using \eqref{eq:qbar:=} we get the transformed version of \eqref{eq:minP-quad-quad-result-0} as,
\begin{equation}\label{eq:minP-quad-quad-result-0-symm}
W(\overline{q},t) \left[
         \begin{array}{c}
           u_1 \\
           u_2 \\
         \end{array}
       \right] = -\sigma \left[
                          \begin{array}{cc}
                           \partial^2_{q_1}L(q_1, q_2) & \big[\partial_{q_1}g(q_1)\big]^T \\[0.5em]
                            \partial_{q_1}g(q_1) &  0 \\
                          \end{array}
                        \right] \left[
         \begin{array}{c}
           \overline{q}_3 \\
           q_4 \\
         \end{array}
       \right], \qquad \sigma > 0
\end{equation}
%=======================================
\begin{proposition}[\cite{rossJCAM-1}]\label{prop:Newton-derived}
Let $K_S(q_1, q_2)$ denote the block $2\times 2$ symmetric matrix on the right-hand side of \eqref{eq:minP-quad-quad-result-0-symm}.  Assume this saddle-point matrix is nonsingular along the iterates of $q$. Suppose we choose $W(\overline{q},t)$ to be the square of $K_S(q_1, q_2)$,
\begin{equation}\label{eq:W4sqp}
W(\overline{q},t):= K_S^2(q_1, q_2)
\end{equation}
Then the resulting inverse-optimal algorithm for solving Problem~$(N_0)$ is equivalent to a sequential quadratic programming (SQP) method with \eqref{eq:minP-quad-quad-result-0-symm} constituting a (damped) Newton method for solving the QP subproblem.
\end{proposition}
%-----------------------
\begin{proof}
Substituting \eqref{eq:W4sqp} in \eqref{eq:minP-quad-quad-result-0-symm} we get,
\begin{equation}\label{eq:Newton4sqp-derived}
\left[
\begin{array}{cc}
\partial^2_{q_1}L(q_1, q_2) & \big[\partial_{q_1}g(q_1)\big]^T \\[0.5em]
\partial_{q_1}g(q_1) &  0 \\
\end{array}
\right] \left[
         \begin{array}{c}
           u_1 \\
           u_2 \\
         \end{array}
       \right] = -\sigma  \left[
         \begin{array}{c}
           \overline{q}_3 \\
           q_4 \\
         \end{array}
       \right], \qquad \sigma > 0
\end{equation}
%---------------------
Applying Newton's method to the necessary conditions of Problem~$(N_0)$ we get,
\begin{equation}\label{eq:Newton4sqp-std}
\left[
\begin{array}{cc}
\partial^2_{q_1}L(q_1, q_2) & \big[\partial_{q_1}g(q_1)\big]^T \\[0.5em]
\partial_{q_1}g(q_1) &  0 \\
\end{array}
\right] \left[
         \begin{array}{c}
           q_1^* - q_1 \\
           q_2^* - q_2 \\
         \end{array}
       \right] = - \left[
         \begin{array}{c}
           \overline{q}_3 \\
           q_4 \\
         \end{array}
       \right]
\end{equation}
where $(q_1^*, q_2^*)$ is the value of $(q_1, q_2)$ at the next iterate.  Obviously, \eqref{eq:Newton4sqp-derived} is identical to \eqref{eq:Newton4sqp-std} under the interpretation $(u_1, u_2)/\sigma = (q_1^*-q_1, q_2^* - q_2)$.
\end{proof}
%-----------------------
\subsubsection{SQP as an Inverse Optimal Algorithm}
According to Proposition~\ref{prop:Newton-derived}, an SQP algorithm may be interpreted as an inverse optimal algorithm for solving Problem~$(N_0)$ under a choice of a quadratic SLF given by \eqref{eq:SLF=quad-quad} and a Riemannian metric stipulated by \eqref{eq:W4sqp}.  This ``SQP metric'' is over the tangent space defined by the $(q_1, q_2)$ coordinates and is given by the square of the KKT matrix. Note however that we did not seek to derive an SQP algorithm.  It is simply a happenstance of selecting a quadratic SLF and the Riemannian metric given by \eqref{eq:W4sqp}. In other words, an inverse optimal algorithm is completely defined once an SLF $S(q,t)$ and a control space $\U(q,t)$ are specified; however, the precise form of the resulting inverse optimal algorithm are not known until the consequences of the choice of the pair $\big(S(q,t), \U(q,t)\big)$ are analyzed.  In this context, Proposition~\ref{prop:Newton-derived} is an ``a posteriori derivation'' of an SQP algorithm using the theory of inverse optimality.

\subsubsection{SQP Versus an Inverse Optimal Algorithm}
Proposition~\ref{prop:Newton-derived} interprets the SQP algorithm in terms of metric spaces.  The concept of using metric spaces for optimization was first proposed by Davidon in 1959\cite{davidon}.  Note, however, that Davidon's ``variable metric'' is based on just the (inverse of the) Hessian and not the square of the KKT matrix.  A few other technical differences between an inverse optimal algorithm and a classical SQP algorithm are as follows:
\begin{enumerate}
\item In an inverse optimal algorithm, one solves the subproblem $(P_0)$.  The cost function in this subproblem is linear while the constraint is quadratic. In contrast, a classical SQP is a quadratic cost problem with a linear constraint.
\item A direct solve of Problem~$(P_0)$ involves choosing $\Delta(q,t) > 0$. If Problem~$(P_0)$ is solved via \eqref{eq:Newton4sqp-derived} then $\Delta(q,t)$ need not be chosen a priori.  Instead, the parameter $\sigma > 0$ may be absorbed in the constraints of  Problem~$(M_h)$ as part of the control jump parameter $h =h_k$ to generate the next iterate. Because there are no guarantees that the product $\alpha_k:= h_k \sigma$  will be equal to unity, it follows that Proposition~\ref{prop:Newton-derived} does not imply convergence for a unit Newton step.
\item Proposition~\ref{prop:Newton-derived} does indeed imply convergence if the next iterate is chosen as part of the inverse optimality process of dissipating the SLF given by \eqref{eq:SLF=quad-quad}.  This dissipation is quantized by the process of solving Problem~$(M_h)$.
\end{enumerate}

\subsection{Choosing $W(q,t)$, Part II}

Let $K_A(q_1, q_2)$ denote the block $2\times 2$ asymmetric matrix on the right-hand side of \eqref{eq:minP-quad-quad-result-0}.  A somewhat surprising aspect of this asymmetric matrix is that it can be positive semidefinite even though its symmetric counterpart (i.e., $K_S(q_1, q_2)$) is indefinite\cite{saddlePt-05,saddlePt-06, saddlePt-08}.
%
%======================
\begin{proposition}\label{prop:AHU-derived}
Let $K_A(q_1, q_2)$ denote the block $2\times 2$ asymmetric matrix on the right-hand side of \eqref{eq:minP-quad-quad-result-0}.  Assume $K_A(q_1, q_2)$ is positive stable. Suppose we choose $W(q,t)$ in \eqref{eq:minP-quad-quad-result-0} according to,
\begin{equation}\label{eq:W4AHU}
W(q,t):= K_A(q_1, q_2)
\end{equation}
and set $\sigma = 1$.
Then \eqref{eq:minP-quad-quad-result-0} reduces to the Arrow-Hurwicz-Uzawa flow corresponding to a differential-equation method for solving Problem~$(N_0)$.
\end{proposition}
%-------------------
\begin{proof}
Substituting \eqref{eq:W4AHU} in \eqref{eq:minP-quad-quad-result-0} and setting $\sigma = 1$ we get,
\begin{equation}\label{eq:AHU-proof-1}
 \left[
         \begin{array}{c}
           u_1 \\
           u_2 \\
         \end{array}
       \right] =  \left[
         \begin{array}{c}
           q_3 \\
           q_4 \\
         \end{array}
       \right]
\end{equation}
Substituting \eqref{eq:B_0-def} and \eqref{eq:integrals-of-motion} in \eqref{eq:AHU-proof-1}  we get,
\begin{equation}\label{eq:AHU-proof-2}
 \left[
         \begin{array}{c}
           \dot q_1 \\
           \dot q_2 \\
         \end{array}
       \right] =  \left[
         \begin{array}{c}
           -\partial_{q_1}L(q_1, q_2) \\
           g(q_1) \\
         \end{array}
       \right] =  \left[
         \begin{array}{r}
           -\partial_{q_1}L(q_1, q_2) \\
           \partial_{q_2}L(q_1, q_2) \\
         \end{array}
       \right]
\end{equation}
which is the Arrow-Hurwicz-Uzawa equation\cite{AHU-flow-2019,AHUY-flow-2010,He-2022,luenberger-2008}.
\end{proof}
%=========================
Although Proposition~\ref{prop:AHU-derived} has derived the Arrow-Hurwicz-Uzawa flow in the spirit of Proposition~\ref{prop:Newton-derived}, note that $K_A(q_1, q_2)$ serves as a pseudometric and not a metric (as erroneously assumed in \cite{rossJCAM-1}). As a result, $\U(q,t)$ as defined in \eqref{eq:U=quad} is not compact and hence $\sigma$ in \eqref{eq:minP-quad-quad-result-0} may be unbounded. The statement of Proposition~\ref{prop:AHU-derived} is based on setting $\sigma = 1$ (or, equivalently, selecting $\sigma$ to be a bounded number) in order to match the Arrow-Hurwicz-Uzawa equation.  In other words, the convergence of the Arrow-Hurwicz-Uzawa flow is not guaranteed by Proposition~\ref{prop:AHU-derived}.  Because the Hamilton-Jacobi formalism (upon which Proposition~\ref{prop:AHU-derived} is based) is a sufficient but not a necessary condition, it is possible to prove convergence without the aid of an SLF by considering \eqref{eq:AHU-proof-1} as a control ansatz in the fundamental dynamical equations $\dot q = F(q,u)$.  In this alternative process, we substitute the control ansatz in \eqref{eq:A_0-def} and get,
\begin{equation}\label{eq:AHU-proof-3}
\left[
         \begin{array}{c}
           \dot q_3 \\
           \dot q_4 \\
         \end{array}
       \right] = - \left[
                          \begin{array}{cc}
                           \partial^2_{q_1}L(q_1, q_2) & \big[\partial_{q_1}g(q_1)\big]^T \\[0.5em]
                            -\partial_{q_1}g(q_1) &  0 \\
                          \end{array}
                        \right] \left[
         \begin{array}{c}
           q_3 \\
           q_4 \\
         \end{array}
       \right] = - K_A^T(q_1, q_2) \left[
         \begin{array}{c}
           q_3 \\
           q_4 \\
         \end{array}
       \right]
\end{equation}
%=======================================
%
Because $K_A(q_1, q_2)$ is positive stable (by assumption) \eqref{eq:AHU-proof-3} converges to its equilibrium point\cite{hnw-ode}, namely a candidate optimal point for Problem~$(N_0)$ defined by the set $T_0$ given by  \eqref{eq:T0-def}.  Hence, we have the following result:
\begin{proposition}
Let the assumptions of Proposition~\ref{prop:AHU-derived} hold. Then the Arrow-Hurwicz-Uzawa flow converges to its equilibrium point.
\end{proposition}
In recognizing that \eqref{eq:AHU-proof-3} is based on the more fundamental physics of optimization as articulated by Theorem~\ref{thm:dynamics4opt}, we get the following result for linear programming problems:
%
%====================
\begin{theorem}\label{thm:AHU-diverges}
The Arrow-Hurwicz-Uzawa algorithm is not convergent for linear programming problems.
\end{theorem}
%----------------
\begin{proof}
For linear programming problems, the Hessian of the Lagrangian $\partial^2_{q_1}L(q_1, q_2) = 0$.  Hence, the symmetric part of $K_A(q_1, q_2)$, namely $(1/2)(K_A + K_A^T)$, is also the zero matrix.  As a result, $K_A^T(q_1, q_2)$ is not positive stable and hence \eqref{eq:AHU-proof-3} does not converge to its equilibrium point.
\end{proof}
%=====================
Theorem~\ref{thm:AHU-diverges} is also proved in He et al.\cite{He-2022} via traditional optimization theory.  He et al. also demonstrate the divergence of the Arrow-Hurwicz-Uzawa algorithm for certain nonlinear optimization problems.  Their results on the divergence of the Arrow-Hurwicz-Uzawa algorithm follow readily from \eqref{eq:AHU-proof-3} and the fact that $K_A(q_1, q_2)$ is not positive stable.

A sufficient condition for $K_A(q_1, q_2)$ to be positive stable is given by the following theorem~\cite{saddlePt-06}:
\begin{theorem}[Benzi-Simoncini]\label{thm:BZ}
Assume $\partial^2_{q_1}L(q_1, q_2)$ is positive definite and $\partial_{q_1} g(q_1)$ has full rank. Define $S_H:= [\partial_{q_1} g(q_1)] [\partial^2_{q_1}L(q_1, q_2)]^{-1} [\partial_{q_1} g(q_1)]^T $ and let $\lambda_{\min}$ denote the smallest eigenvalue of $\partial^2_{q_1}L(q_1, q_2)$.  If $\lambda_{\min} \ge 4 \|S_H\|_2$ then all eigenvalues of $K_A(q_1, q_2)$ are real and positive.
\end{theorem}
%=================
As noted by Benzi and Simoncini\cite{saddlePt-06}, Theorem~\ref{thm:BZ} is a sufficient but not a necessary condition.

\section{Additional Examples with Theoretical Generalizations}
Section~\ref{sec:examples} provides new precision to the examples discussed in \cite{rossJCAM-1}.  Additional examples are provided in \cite{rossJCAM-2} and \cite{ross:CD} with extensions of the theory along different fronts.  In the following we briefly comment on some of these results with further insights.

\subsection{An (Inverse) Optimal Control Theory for Accelerated Optimization}\label{sec:acceleration}
In recent years, there have been a number of attempts to ``explain'' Nesterov's celebrated accelerated optimization method\cite{nesterov83}.  As noted in Section~\ref{sec:prelim-discuss}, barring \cite{rossJCAM-2}, all of these explanations begin with Nesterov's algorithm.  The question posed in \cite{rossJCAM-2} is whether accelerated optimization algorithms can be ``derived'' without using any a priori information of Nesterov's equations. An explanation was alluded to in Section~\ref{sec:HJ-theory} by suggesting that one possible measure of optimality for optimal optimization was the total variation of a hidden algorithm primitive.  Consider applying this ``principle'' to the state trajectory $t \mapsto q(t)$.  A standard assumption in optimal control theory\cite{vinter,clarke-2013book} is the regularity condition,  $q(\cdot) \in W^{1,1}\big([t_0, t_f], \real{2(1+n+m)}\big)$.  A simple engineering approach\cite{ross-book} to reduce the total variation of a state variable is to seek $q(\cdot) \in W^{2,\infty}\big([t_0, t_f], \real{2(1+n+m)}\big)$. For $q_1(\cdot)$, the functional constraint $q_1(\cdot) \in W^{2,\infty}\big([t_0, t_f], \real{n}\big)$ can be directly imposed in Theorem~\ref{thm:dynamics4opt} by setting,
\begin{equation}\label{eq:accel-idea}
\dot q_1(t) := v_1(t), \quad \dot v_1(t) := u_1(t)
\end{equation}
so that a bounded $u_1(t)$ steers the ``acceleration'' $\dot v_1(t)$. Linking inverse-optimality to partial guidability per Section~\ref{sec:inv-opt-hid-prim} generates a final-time equilibrium condition for the $v_1$ variable given by $v_1(t_f) = 0$.  Carrying out the remainder of the steps developed in Section~\ref{sec:inv-opt-theory} generates an inverse optimal algorithm\cite{rossJCAM-2} which turns out to be equivalent to Nesterov's accelerated gradient method.  Note, however, that this process of generating an accelerated gradient method is not based on seeking Nesterov's method; rather, it is a natural consequence of producing equations for hidden algorithm primitives with reduced total variations based on the ``intuition'' suggested by \eqref{eq:accel-idea}.  Furthermore, in much the same way as Propositions~\ref{prop:Newton-derived} and \ref{prop:AHU-derived} were not based on deriving an SQP and the Arrow-Hurwicz-Uzawa algorithms respectively, the results presented in \cite{rossJCAM-2} are not based on deriving Nesterov's method; rather, it is shown that Nesterov's method is a consequence of \eqref{eq:accel-idea}, inverse optimality and a quadratic SLF. Indeed, there is no a priori assumption of Nesterov's equations in this process and hence the results of \cite{rossJCAM-2} may be viewed as an a posteriori ``derivation'' of Nesterov's accelerated gradient method (with caveats similar to those that accompany Propositions~\ref{prop:Newton-derived} and \ref{prop:AHU-derived}).

\subsection{Nonsmooth SLFs and Nonsmooth Optimization}\label{sec:nonsmooth-stuff}
In \cite{ross:CD}, nonsmooth SLFs are used to derive several variants of coordinate descent algorithms.  This ``derivation'' is achieved in the same spirit as  Propositions~\ref{prop:Newton-derived} and \ref{prop:AHU-derived} by using a collection of nonsmooth max functions for the SLFs. When a weak Lyapunov function (such as an SLF) is nonsmooth, several variants of the infinitesimal decrease condition (Cf.~\eqref{eq:diss-0}) based on different types of subgradients have been proposed in the controls literature\cite{clarke-nolcos-2010,osinenko-2020,clarkeLyap}. In \cite{ross:CD}, coordinate descent algorithms for unconstrained optimization are generated based on the following infinitesimal decrease condition,
\begin{equation}\label{eq:max-min4SLF}
\max_{\zeta \in \partial S(q_a)} \min_{u \in \U(q,t)} \langle \zeta, f(q_0, q_1, q_2, u)  \rangle + R(q,t) \le 0 \quad \text{for } q_a \not\in T
\end{equation}
where $R(q,t)$ is the rate function as defined before (Cf.~\eqref{eq:Sdot+R<0}).
As a new illustration of this process, consider an unconstrained optimization problem (i.e., $m=0$ in Problem~$(N)$).  Suppose we choose an SLF given by the $\ell_1$-norm of $q_3$ so that $S(q) := \|q_3\|_1$.  Suppose further that we choose a control space similar to  \eqref{eq:U=quad} with $W(q,t) := \partial^2_{q_1} g_0(q_1)$.  Then, following the same procedure as in \cite{ross:CD}, it is straightforward to show that the resulting inverse-optimal algorithm is a sign gradient descent given by,
\begin{equation}\label{eq:signGrad}
q_1(k+1) = q_1(k) - \alpha_k\ \text{sign}\left[\partial_{q_1}g_0(q_1(k))\right], \quad \alpha_k > 0
\end{equation}
Note once again that the derivation of \eqref{eq:signGrad} does not involve producing a differential equation or a hidden algorithm primitive.  Furthermore, it follows from the results of Section~\ref{sec:inv-opt-algol} that $\alpha_k$ in \eqref{eq:signGrad} must be chosen to decrement the (nonsmooth) SLF given by $S(q) = \|q_3\|_1$.

The sign gradient descent is an invention of recent origin\cite{sign-properties,sign-compression} in that it was motivated by the needs of distributed computing for large $n$ in machine learning applications.  That is, by transmitting only the sign of the gradient across distributed processing units, the communication bottleneck is alleviated for a faster learning process\cite{sign-compression,GPU-Nature}.  Thus, the sign gradient descent algorithm was invented out a practical necessity and not via the theory of inverse optimality.  Nonetheless, it can now be explained in terms of the new physics of optimization as an algorithm that is generated by dissipating the $\ell_1$-norm of the gradient of the objective function under a Riemannian metric given by the Hessian (of the objective function).

The infinitesimal decrease condition given by \eqref{eq:max-min4SLF} can be further weakened using a Dini derivate of the SLF\cite{clarke-nolcos-2010,osinenko-2020}.  In this case, the max operation in \eqref{eq:max-min4SLF} can be replaced by the minimum\cite{clarke-nolcos-2010}:
\begin{equation}\label{eq:min-min4SLF}
\min_{\zeta \in \partial S(q_a)} \min_{u \in \U(q,t)} \langle \zeta, f(q_0, q_1, q_2, u)  \rangle + R(q,t) \le 0 \quad \text{for } q_a \not\in T
\end{equation}
It is straightforward to show that replacing \eqref{eq:max-min4SLF} by \eqref{eq:min-min4SLF} in the results of \cite{ross:CD} yields generalized coordinate descent algorithms that more closely follow the Gauss-Southwell rule\cite{GS-rule-2015,GS-rule-2017} where $\zeta \in \partial S(q_a)$ is chosen to maximally decrement the (nonsmooth) max SLFs.

Finally, we briefly expand on the comment alluded to Section~\ref{sec:practialInvOptAlgol} with regards to choosing $\U(q,t)$. Suppose we choose $\U(q,t)$ to be any compact convex set (for all $(q, t)$). Then Problem~$(M_u)$ is a convex optimization problem with a linear objective function. Hence it follows from Theorem~\ref{thm:inv-opt-algol} that this family of inverse-optimal algorithms solves Problem~$(N)$ via a sequence of convex optimization problems. Because Problem~$(M_u)$ is not generated from Problem~$(N)$ by either convexification or linearization as in the case of current sequential convex programming methods\cite{SCP-2010,SCP-revisited-2021,SCP-survey-2021}, the resulting new algorithms constitute a different breed of sequential convex optimization methods based on inverse optimality.

From the preceding discussions and the results of Section~\ref{sec:examples}, it is clear that a vast number of algorithms can be derived or invented using the Hamilton-Jacobi theory of inverse optimality.  The basic tool for the derivation/invention is simply choosing the pair $(S(q,t), \U(q,t))$.

An open question at the present time is the proper way to generalize Theorem~\ref{thm:dynamics4opt} for nonsmooth data functions. It is apparent that this generalization would require the use of second-order subdifferential calculus\cite{boris-book-2006,RockWets-book-2009,borisRock-2012}.  As noted in Rockafellar and Wets\cite{RockWets-book-2009}, page~579, ``The technical and conceptual challenges are formidable, however.'' We do not pretend it to be otherwise. Further discussions on this topic are well beyond the scope of this paper.

\newpage
\section{Conclusions and Outlook}

The results of the present paper along with those presented in \cite{ross:CD,rossJCAM-1,rossJCAM-2} show the following:
\begin{enumerate}
\item When the data functions of a continuous optimization problem are smooth (twice differentiable) Theorem~\ref{thm:dynamics4opt} describes the mathematical physics of hidden algorithm primitives for many number of constrained and unconstrained optimization algorithms.
\item Hidden algorithm primitives for accelerated optimization algorithms, including Nesterov's accelerated gradient algorithm, also satisfy the dynamical equations given by Theorem~\ref{thm:dynamics4opt} with an additional smoothness constraint imposed on the evolution of $t \mapsto q(t)$.  This smoothness condition has the effect of reducing the total variation of $q(\cdot)$ thereby producing ``acceleration.''
\item The generation of algorithms according to the proposed theory of inverse optimality never requires the production and propagation of differential equations. In fact, there are no assurances that such an algorithm even tracks a solution to a differential equation.
\item In much the same way as nonsmooth Lyapunov functions for continuously differentiable dynamics are useful in control theory, nonsmooth SLFs are particularly useful for optimization because inverse-optimal algorithms are produced in terms of controlled jumps where each jump obeys an instantaneous version of a Hamilton-Jacobi inequality.
 \item Unlike traditional optimization theory wherein convergence must be proved for a proposed algorithm, the inverse-optimal framework comes pre-equipped with convergence per Theorem~\ref{thm:inv-opt-algol}. The essence of this theorem is that each controlled jump must decrement the SLF upon which the control was derived. The resulting polygonal arc (i.e., the algorithm) converges to the zero of the SLF, which, by definition, satisfies the necessary conditions of optimality (within some $\epsilon>0$ tolerance).
\item The specific details of a particular inverse-optimal algorithm that is defined by the selection of an SLF and the control space are unknown until the results of these choices are further analyzed.  For example, it us unknown a priori that using a quadratic SLF and a particular Riemannian metric generates an SQP method.
\end{enumerate}
The notion that optimization algorithms can be derived from a small set of universal equations sounds preposterous.  Nonetheless, once the idea of a hidden algorithm primitive is postulated, the result is less surprising after the (control) dynamical equations are derived and optimal control theory is invoked.  What is indeed surprising is that it is eventually not necessary to ever generate a differential equation to produce an algorithm. At the same time, the continuous-time concept of a hidden algorithm primitive is crucial to an understanding of the natural physics of optimization.

Because of the close connection between the Hamilton-Jacobi and Schr\"{o}dinger equations~\cite{PNAS-2013,Field-2011}, it may be possible to transform \eqref{eq:HJ} to quantum mechanical equations using the wave ansatz and a Madelung-Bohm-type quantum potential.
If this transformation is possible, then the resulting Schr\"{o}dinger equation becomes the fundamental equation for optimal algorithm primitives.  Consequently, such a Schr\"{o}dinger-equation-based algorithm can be directly implemented on a quantum computer for solving Problem~$(N)$.  In other words, it may be more efficient to implement a Schr\"{o}dinger-equation-based algorithm on a quantum machine than one based on simulating Hamilton-Jacobi equations.   If such a transformation is possible, the main use of the resulting quantum solver will likely be future machine learning applications that are expected to solve massively large-scale optimization problems that require computational throughput beyond the limits of classical computing devices.

\vskip 6mm
%\noindent{\bf Acknowledgments}
%
%%The second author was supported by the  National Natural Science Foundation  under Grant No.
%%1244145e2.

\end{document}